\documentclass{amsart}

\usepackage{amssymb,amsmath}
\usepackage{graphicx}
\usepackage{enumerate}
\usepackage{colortbl}
\usepackage{cite}
\usepackage{xcolor}

\usepackage{amsrefs}

\usepackage{verbatim}

\usepackage{geometry}
\usepackage{comment}
\usepackage{graphicx,caption,subcaption}
\usepackage{tikz}
\usepackage{multirow}

\usepackage[title,titletoc]{appendix}

\colorlet{linkequation}{blue}
\usepackage[colorlinks,plainpages=true,pdfpagelabels,hypertexnames=true,colorlinks=true,pdfstartview=FitV,linkcolor=blue,citecolor=red,urlcolor=black]{hyperref}
\usepackage{hyperref}

\PassOptionsToPackage{unicode}{hyperref}
\PassOptionsToPackage{naturalnames}{hyperref}

\definecolor{dgreen}{rgb}{0,0.5,0}
\definecolor{violet}{rgb}{0.5,0,0.5}
\definecolor{dred}{rgb}{0.7,0,0}
\definecolor{ddred}{rgb}{0.5,0,0}
\definecolor{dblue}{rgb}{0,0,0.5}
\definecolor{ddblue}{rgb}{0,0,0.3}
\definecolor{llgray}{rgb}{0.9,0.9,0.9}
\definecolor{lgray}{rgb}{0.7,0.7,0.7}

\newtheorem{defn}{Definition}[section]
\newtheorem{definition}[defn]{Definition}
\newtheorem{lemma}[defn]{Lemma}
\newtheorem{proposition}[defn]{Proposition}
\newtheorem{theorem}[defn]{Theorem}
\newtheorem{assumption}[defn]{Assumption}

\newtheorem{remark}[defn]{Remark}
\numberwithin{equation}{section}

\newcommand{\ep}{{ \epsilon  }}

\newcommand{\bq}{\begin{equation}}
\newcommand{\eq}{\end{equation}}

\newcommand{\R}{{ \mathbb R}^n}

\newcommand{\bke}[1]{\left( #1 \right)}

\newcommand{\bket}[1]{\left\{ #1 \right\}}

\newcommand{\abs}[1]{\left| #1 \right|}

\newcommand{\om}{{ \omega  }}

\newcommand {\la}{\lambda}
\newcommand{\na}{\nabla}

\newcommand {\al}{\alpha}
\newcommand {\be}{\beta}

\newenvironment{pfthm-NS}{{\par\noindent
            \textbf{Proof of Theorem \ref{maintheorem(NS)}}\quad}}{\hfill $\square$}
            
\newenvironment{pfthm-SS-19}{{\par\noindent
            \textbf{Proof of Theorem \ref{theorem260403}}\quad}}{\hfill $\square$}

\newcommand{\Rn}{{\mathbb R}^{n-1}}
\newcommand{\De}{\Delta}
\newcommand{\te}{\theta}
\newcommand{\de}{\delta}
\newcommand{\Ga}{\Gamma}

\AtBeginDocument{%
   \def\MR#1{}
}
\nocite{*}

\begin{document}

\title[Boundary Layer Separation for Incompressible Fluid Flow ]{On the Existence of Boundary Layer Separation for Incompressible Fluid Flow in the Half-Space}
%Existence of weak solutions for nonlinear diffusion equations with drift for measure data
%
%Existence of weak solutions for nonlinear diffusion equations with measure data and divergence type of drift terms

\author{TongKeun Chang}
\address{TongKeun chang: Department of
Mathematics, Yonsei University, 50 Yonsei-ro, Seodaemun-gu, Seoul,
South Korea 120-749 } \email{chang7357@yonsei.ac.kr }

\author{KyungKeun Kang}
\address{KyungKeun Kang : Department of
Mathematics, Yonsei University, 50 Yonsei-ro, Seodaemun-gu, Seoul,
South Korea 120-749 } \email{kkang@yonsei.ac.kr }

\thanks{  }

\date{}

\makeatletter
\@namedef{subjclassname@2020}{
  \textup{2020} Mathematics Subject Classification}
\makeatother
\subjclass[2020]{35Q30, 35B50}

\begin{abstract}
We consider the Stokes system in the half-space with localized boundary data. We prove that a boundary layer separation point exists provided that a certain singular integral determined by the boundary data is negative. On the other hand, if this integral is strictly positive, then boundary layer separation does not occur.
When boundary layer separation occurs, we also investigate the dynamics of the separation point and the sign of the pressure gradient. Furthermore, by a perturbation argument, we construct solutions to the Navier--Stokes equations in the half-space that exhibit the same qualitative behavior as in the Stokes case.
\end{abstract}
\keywords{Existence of boundary layer separation and adverse pressure gradient,    Stokes system, Navier-Stokes system, half space }
\maketitle
%\tableofcontents

\section{Introduction}

In this paper, we consider the Stokes system in the half space.
\begin{equation}\label{Stokes-10}
w_t - \De w+  \na p =0\qquad \mbox{div } w =0 \quad \mbox{ in }
\, {\mathbb R}^n_+ \times (0, T),
\end{equation}
where  $ w$  indicates the velocity of the fluid and $p$ is the associated pressure  
with zero initial data and non-zero boundary data:
\begin{equation}\label{Stokes-icbc}
\left. w \right|_{t=0}=0,
\qquad
\left. w \right|_{x_n=0}=g.
\tag{1.2}
\end{equation}
Here,  the boundary data $g:\mathbb{R}^{n-1}\times [0, \infty) \to \mathbb{R}^n$
are of separable form and have  compact support in the spatial and temporal variables, which will be specified later.

Previous works have established the existence of singular solutions to the Stokes equations whose velocity gradients are unbounded near the boundary, away from the support of the boundary data, while both the velocity and its gradient remain locally square-integrable (see \cite{Kang05,CK20,KangTsai22FE,CK24,KM}; see also \cite{CH2,CCK2,KangTsai22} for related works).
An analogous result was obtained in \cite{CK25} for the case of a nonzero localized external force, rather than nonzero boundary data. More recently, it was shown in \cite{CK26} that, due to a nonlocal effect, the velocity itself may become unbounded away from the support of the boundary data even under the no-slip boundary condition.

In this paper, we focus on the sign change of the velocity field near the boundary. Assuming that the boundary data are sufficiently regular, this phenomenon implies that the normal derivative of the velocity vanishes and leads to the formation of a boundary layer separation point (see Definition \ref{def-blsp} below).

We briefly review previous studies on boundary layer separation. To investigate this phenomenon, Prandtl \cite{P} derived an asymptotic form of the Navier–Stokes equations near the boundary in a two-dimensional half-plane, namely the Prandtl equations, which are written as follows.
\[
\partial_t u+u\partial_x u+v\partial_y v-\partial_y^2 u=-\partial_x p_E,\qquad \partial_x u+\partial_y v=0
\]
with initial and boundary conditions
\[
u|_{t=0}=u_0,\quad u|_{y=0}=0,\quad \displaystyle\lim_{y\rightarrow \infty}u(x,y,t)=u_E(x,t).
\]
Here \(u_E\) and \(p_E\) denote the tangential velocity and the pressure
of the outer Euler flow evaluated at the boundary. In particular, they satisfy
the Bernoulli relation
$\partial_t u_E + u_E \partial_x u_E = -\partial_x p_E $.
Thus the pressure gradient appearing in the Prandtl equation is determined
by the corresponding outer Euler flow.

The steady Prandtl theory exhibits a sharp distinction between favorable and
adverse pressure gradients.  For \(d p_E/dx\leq 0\), Oleinik \cite{Oleinik}
proved global-in-\(x\) well-posedness, whereas for \(d p_E/dx>0\) solutions may
separate at a finite streamwise position \(x^*>0\).  At separation the wall
shear degenerates,
$\partial_y u(x,0)\to 0$
as $x\nearrow x^*$.
which corresponds to the classical Goldstein singularity \cite{Goldstein}.
This singular behavior, including the vanishing rate of the wall shear, was
rigorously analyzed in \cite{Dali-Mas} (see also \cite{SWZ}).  The
post-separation steady regime is still much less understood, with recent
progress mainly concerning flow-reversal profiles such as Falkner--Skan
solutions \cites{masreversal,dalireversal}.  

In the unsteady case, numerical
studies \cites{vanshen1,vanshen2} led to the conjecture that finite-time
separation is accompanied by singularity formation.  This was proved for a
trivial outer flow under symmetry assumptions in \cite{eng}, and for the
nontrivial Euler trace of \cite{vanshen1} in \cite{kvw}, via a convexity
argument involving a Lyapunov functional related to the displacement thickness (see also \cites{unsteadyinvisicid,unsteadyviscous}).

Recently, the triple-deck model has been proposed as a refined description of the flow near separation, and substantial progress has been made on its mathematical analysis, including local well-posedness and ill-posedness results (see \cites{Iyervicol, iyermaekawa, Dietertvaret, Varetiyer} and the references therein).

Our main objective of this paper is to prove the existence of a boundary layer
separation point for the Stokes system in the half-space.  A related flow-reversal
phenomenon was established in \cite{CKM26}, where it was shown that each
component of the velocity field changes sign at least once.  This result suggests
that even the linear Stokes system may exhibit reversal mechanisms connected to
boundary layer separation.  In \cite{CKM26}, however, the localized nonzero
boundary data are singular in the time variable, and the normal derivative of
the tangential velocity becomes unbounded at the singular time around which the
sign change occurs.  In contrast, the present work treats sufficiently regular
localized boundary data, for which the normal derivatives of the velocity field
remain bounded.  We then establish the existence of a boundary-layer separation
point under appropriate assumptions on certain time-integral quantities
determined by the boundary data.

The main tool in our analysis is the representation formula for solutions to the Stokes system in the half-space with nonzero boundary data (see \eqref{rep-bvp-stokes-w} and \eqref{Poisson-tensor-K}). Using this representation formula together with the identity \eqref{1006-3}, we decompose the solution into five terms, as in \eqref{repre0926}. From this decomposition, we derive a formula for the normal derivative of the tangential components of the velocity on the boundary (see Lemma~\ref{thoe1021-1}).
Since the tangential components of the boundary data vanish and only the normal component is nonzero, this formula simplifies considerably and reduces to the sum of two terms; see \eqref{shortg}. At this stage, the integral quantity \eqref{m-t}, which is associated with the temporal part of the boundary data, turns out to play a crucial role in determining the sign of the normal derivative. One of the main steps in the analysis is to show that the sign of this integral determines whether boundary layer separation occurs (see Theorem \ref{maintheorem}).

In the case where a boundary layer separation point exists, it is not stationary; rather, it moves away from the origin and tends to infinity in finite time. In fact, the time at which the boundary layer separation point escapes to infinity can also be characterized in terms of the integral (see Theorem \ref{theorem1221}). This characterization is also obtained through a careful analysis of the representation formula for the normal derivative on the boundary.
Furthermore, if the temporal part of the boundary data is unimodal, namely, increasing and then decreasing to zero, then the boundary layer separation point is unique at each time for which it exists (see Theorem~\ref{theorem260403}).

As for the pressure, using the representation formula \eqref{representationpressure}, we prove that the tangential derivatives of the pressure undergo a sign change. In particular, if the boundary data are decreasing in time, then the pressure is strictly increasing in the tangential spatial direction (see Theorem \ref{pressure1103}).

Finally, we show that  similar phenomena can be observed for the Navier-Stokes equations, namely
\begin{equation}\label{nse-30}
u_t - \Delta u + (u\cdot\nabla)u + \nabla q = 0,
\qquad
\operatorname{div} u = 0
\quad \text{in } \mathbb{R}^n_+ \times (0,T),
\end{equation}
together with the initial and boundary conditions \eqref{Stokes-icbc}. Using a perturbative argument, we construct solutions that exhibit behavior analogous to that of the Stokes system (see Theorem~\ref{maintheorem(NS)}).

%For convenience, we say that $(x', t) \in \mathbb{R}^{n-1} \times (0, \infty)$ satisfies the condition ${\mathcal P}^i_{+}$ (or ${\mathcal P}^i_{-}$) if there exists $\epsilon(x', t) > 0$ such that $w_i(x', x_n, t) > 0$ (or $w_i(x', x_n, t) < 0$) for all $0 < x_n < \epsilon$.

Boundary-layer separation occurs at a point on the boundary where the near-wall
flow moves in the downstream direction up to that point and begins to reverse
thereafter.  Motivated by this physical picture, we introduce a mathematical
definition of a boundary-layer separation point.
For convenience, we first introduce the following notation.  

We say that
$(x',t)\in \mathbb{R}^{n-1}\times (0,\infty)$
satisfies the condition \({\mathcal P}^i_{+}\) respectively, \({\mathcal P}^i_{-}\),
if there exists \(\epsilon=\epsilon(x',t)>0\) such that
\[
w_i(x',x_n,t)>0
\quad
\text{respectively,}\quad
w_i(x',x_n,t)<0
\]
for all \(0<x_n<\epsilon\).
We now define a boundary-layer separation point as follows.

\begin{definition}\label{def-blsp}
Let $1 \leq i \leq n$. 

\begin{itemize}
\item[(1)]
We say that $(x', t) \in \mathbb{R}^{n-1} \times (0, \infty)$ satisfies the condition ${\mathcal P}^i_{+}$ (or ${\mathcal P}^i_{-}$) if there exists $\epsilon(x', t) > 0$ such that $w_i(x', x_n, t) > 0$ (or $w_i(x', x_n, t) < 0$) for all $0 < x_n < \epsilon$.

\item[(2)]
We say that $(x', t^*_i)$ is a \textbf{boundary layer separation point} of $w_i$ if the following hold:
\begin{enumerate}
    \item[(i)] For all $0 < t < t^*_i$, $(x', t)$ satisfies the condition ${\mathcal P}^i_+$.
    \item[(ii)] There exists a sequence $\{t_k\}_{k=1}^\infty$ such that $t_k > t^*_i$, $\lim_{k \to \infty} t_k = t^*_i$, and each $(x', t_k)$ satisfies the condition ${\mathcal P}^i_-$.
\end{enumerate}
\end{itemize}
\end{definition}

We will assume that the fluid is flowing into  $\R_+$ from the origin on the boundary. Considering this physical situation, we specify  boundary data $g:\mathbb R^{n-1}\times  (0, \infty) \rightarrow  \R$ with only non-zero $n$-th component defined
as follows:
\begin{align}\label{0502-6}
g(y', s)  = a \psi (|y'|)\phi(s)\textbf{e}_n, \quad \psi \in C^\infty_c (B'_1), \quad  \psi \geq 0, \quad \psi  >  1 \quad \mbox{in} \quad B'_\frac12, 
\end{align}
where $a> 0$ is a small positive number determined later on for existence of solution to  Navier-Stokes equations.

Based on the physical observation that a decrease in fluid velocity after an initial increase creates an adverse pressure gradient leading to a boundary layer separation point, we mathematically assume $\phi$ as follows  (see \cite{P}, \cite{SG1},  \cite{SG}, \cite{St}).   
\begin{assumption}\label{assume1}
Let $\phi \in C^{\frac12 +\frac{\al}2} ([0, 2 ) )$ for some $ 0 < \al  <1$,  and $ \phi$ be monotonically increasing in $(0,1)$ and monotonically decreasing in $(1,2)$ with $\phi(0) =0$.   
\end{assumption}

\begin{remark}
In the context of the no-slip boundary condition, it is natural to examine the normal derivative to determine the sign of the fluid velocity near the boundary. 
Under assumption \eqref{0502-6} for $g$ and the regularity of $\phi$ from  Assumption \ref{assume1}, the normal derivative of the tangential parts $w_i$  of the solution $w$ to \eqref{Stokes-10} with boundary condition $g$ is given by \eqref{shortg}. It turns out that from \eqref{shortg} that the sign of the normal derivative $D_{x_n} w_i (x', 0, t)$ is primarily determined by the sign of the following integral, for convenience, denoted by:
\begin{align}\label{m-t}
M(t) := \int_{-\infty}^t \frac{\phi(t) - \phi(s)}{(t - s)^{3/2}} \, ds.
\end{align}
\end{remark}

 \begin{assumption}\label{assume2}
 Let $\phi$ satisfy Assumption \ref{assume1} and $M(t)$ be given in \eqref{m-t}.    There exist  $ t_0, \, t_1 \in (1,2)$  with $ t_0 < t_1$  such that $ M(t) <0$ for all $ t_0 < t < t_1$.
 \end{assumption}

\begin{assumption}\label{assume3}
 Let $\phi$ satisfy Assumption \ref{assume1}  and $M(t)$ be given in \eqref{m-t}. There exists $ c_0> 0$ such that  $M(t) > c_0$ for all $ 0 < t <2$.
 \end{assumption}

We provide the conditions of $\phi$ satisfying Assumption \ref{assume2} and Assumption \ref{assume3} in Appendix \ref{proofprop0320-2}.

Now, we are ready to state our resuts.
First, we show that an adverse pressure gradient develops. More precisely, the tangential derivative of the pressure is initially negative and becomes strictly positive at a later time.  Let $i=1,2, \cdots, n$ and, for convenience, for a  fixed $x'\in \mathbb R^{n-1}$, we define $t^*_{p,i} (x')$ by
 \begin{align}\label{t-p*}
t^*_{p,i}(x')   = \sup \bket{ 0 < t< 2 \, | \, D_{x_i} p(x',0,s) <0 \quad \mbox{for all } s\in (0, t)}.
\end{align}

\begin{theorem}\label{pressure1103}
Let  $w = (w_1, \cdots, w_n)$ be a solution   of   system  \eqref{Stokes-10} with boundary data $g$, where $ g$ is defined by \eqref{0502-6}.  Let $p$ be the  corresponding pressure, and let $t_{p,i}^* (x')$  be as in \eqref{t-p*}.
If $\phi \in C^1(0,1)$ satisfies Assumption \ref{assume1}, then the following holds for each $1 \leq i \leq n-1$.
 \begin{itemize}
\item[(i)]
%Let $ 1 \leq i \leq n-1$ and $x'\in \mathbb R^{n-1}$. 
Let $x'\in \mathbb R^{n-1}$.
Assume that for some $c>0$ and $t_1 \in (0,1)$, $ \phi(t) \leq c  t \phi'(t) $ for all $ t\in (0, t_1).$
%Suppose that $    \phi(t) \leq c  t \phi'(t) $ for $ 0 <  t < t_1$ for some $ 0 < t_1 < 1$.  
There exist constants  $c_{*4}>2$ and $ c_1> 0$ such that if $ c_{*4} < |x'| < c_1 x_i$, then $t_1 < t_{p,i}^* (x')$.
%\begin{align}
%t_1 < t_{p}^* (x'),
%\end{align}
%where $t_{p}^* (x')$ is given in \eqref{t-p*}.
\item[(ii)]
%Let $ 1 \leq i \leq n-1$. 
If $ \phi'(t_2) =0 $ for some $ 0 < t_2 < 1$, then, $t_{p,i}^* (x') \leq  t_2$.
%\begin{align}
%t_{p}^* (x') \leq  t_2. 
% \end{align}
 In particular, $ t_{p,i}^* (x') \leq 1$.
 
 \item[(iii)]
 %Let $ 1 \leq i \leq n-1$. 
 %Then, 
 $D_{x_i} p(x', 0, t) > 0$ for $ 1 < t <2$.
% \begin{align}
% D_{x_i} p(x', 0, t) > 0.
% \end{align}

\end{itemize}

\end{theorem}

The next theorem concerns the existence of a point of boundary layer separation. The crucial observation is that the existence or nonexistence of such a point depends on the sign of the integral \eqref{m-t}. More precisely, it can be stated as follows.

\begin{theorem}\label{maintheorem}
Let $w = (w_1, \dots, w_n)$ be a solution to the system \eqref{Stokes-10} with boundary data $g$ defined by \eqref{0502-6}. For each $1 \leq i \leq n-1$ we have the following.
\begin{itemize}

\item[(i)]
Suppose that $\phi$ satisfies Assumption \ref{assume2}. Then there exist constants $c_{*1}>0$ and $c_{1i}>0$ such that, for any $x'$ satisfying $c_{*1}<|x'|<c_{1i}x_i$, there exists $t_i^*(x')\in(1,t_0)$ such that $(x',0,t_i^*(x'))$ is a point of boundary layer separation for $w_i$. Furthermore, $(x', t_i^*(x'))$ satisfies the condition ${\mathcal P}^i_+$.
  
\item[(ii)]
Suppose that $\phi$ satisfies Assumption \ref{assume3}. Then there exist constants $c_{*2}>0$ and $c_{2i}>0$ such that, for any $x'$ satisfying $c_{*2}<|x'|<c_{2i}x_i$, the velocity component $w_i$ does not exhibit boundary layer separation at any boundary point $(x',0,t)$ for $t\in(0,2)$.

\end{itemize}
\end{theorem}

In the next theorem, under Assumption \ref{assume2}, we investigate the motion of a point of boundary layer separation over time. Roughly speaking, such a point moves away from the origin and reaches infinity in finite time. To be  more precisely, we obtain the following.

\begin{theorem}\label{theorem1221}
Let $\phi \in C^{1 +\frac{\al}2} ([0,2])$ satisfy Assumption \ref{assume2}.  For each $1 \leq i \leq n-1$ we have the following.
\begin{itemize}
\item[(i)]
There exists a constant $c_{*3}>1$ such that, for any $x_1,x_2\in\mathbb{R}$ with $c_{*3}<x_1<x_2$,
\begin{align*}%\label{260105-1}
  t_i^*(x_1{\bf e}_i) < t_i^*(x_2{\bf e}_i),
\end{align*}
where $t_i^*(x_j{\bf e}_i)$ denotes the time at which separation occurs at $x_j{\bf e}_i$, for $j=1,2$, as defined in Theorem \ref{maintheorem}.
 
\item[(ii)]
Let $t_0^*$ be defined by
\begin{align}\label{t^*_1}
t_0^*=\sup\left\{r\in(0,2)\,\middle|\, M(t) >  0 \quad \mbox{for all } t< r \right\}.
\end{align}
Then $t_i^*(x_i{\bf e}_i)<t_0^*$ for all $x_i>c_{*3}$, and
\begin{align*}
\lim_{x_i\to\infty} t_i^*(x_i{\bf e}_i)=t_0^*.
\end{align*}
\end{itemize}
\end{theorem}

In the next theorem, we present another condition in terms of the sign of $\frac{d}{dt}M(t)$, which also implies the occurrence of boundary layer separation for both the tangential and normal velocity components. We show that the normal derivative of the pressure becomes negative after the time at which boundary layer separation occurs for the normal velocity. We also determine the temporal behavior of the corresponding separation point.

We first observe, by integration by parts, if  $ \phi \in C^2([0,2])$ with  $\phi(0) =\phi'(0)  =0$, that
\begin{align}\label{1026-6-1}
\begin{split}
M(t) = \int_{-\infty}^t  \frac{  \phi(t) - \phi(s)  }{(t-s)^\frac32}  ds  
%& =    2\int_{-\infty}^t \big( \phi(t) - \phi(s) \big)  \frac{d}{ds}   (t-s)^{-\frac12}  ds\\
& =   2\int_{0}^t   \phi' (s)     (t-s)^{-\frac12}  ds 
%& =   -4\int_{0}^t   \phi' (s)    \frac{d}{ds} (t-s)^{\frac12}  ds\\
 =    4\int_{0}^t   \phi'' (s)   (t-s)^{\frac12}  ds.
\end{split}
\end{align}
Therefore, it follows via differentiation that   
\begin{align}\label{M-prime}
M'(t) = \frac{d}{dt} M(t) =2 \int_0^t  \phi''(s) (t-s)^{-\frac12}  ds.
\end{align}

We are now ready to state the main result concerning the condition \eqref{M-prime}.
\begin{theorem} \label{theorem260403}
Let $\phi \in C^2([0,2])$ satisfy Assumption \ref{assume1} and $ \phi'(0)  =0$. We assume that $M(2)  < 0$ and  there exists  $ e_0 > 0$   such  that for $ 1 < t < 2$, 
\begin{align}\label{260321-7}
M'(t) < -e_0.
\end{align}
\begin{itemize}
\item[(i)]
There exists $ 2 <  c_{*5}$, for  $c_{*5} < |x'| < c_i x_i $,  $ 1 \leq i \leq n-1$,   $ w_i$ has the unique boundary layer separation $ (x', t^*_i(x'))$ with $ 1 < t^*_i (x') <2$.

\item[(ii)] 
There exists $ 2 <  c_{*6}$, for  $c_{*6} < |x'| $,     $ w_n$ has the unique boundary layer separation $ (x', t^*_n(x'))$  with $ 1 < t^*_n (x') <2$. Moreover,  $D_{x_n} p(x',0,t) > 0$ for $ 0 < t< t_n^* (x')$ and $D_{x_n} p(x', 0,t) < 0$ for $ t_n^* (x') < t< 2$.

\item[(iii)]
There exists $ 0 < c_{*7}$ such that if $ c_* < |x'_1| < |x'_2|$, then
\begin{align*} 
  t_n^*(x'_1  ) < t_n^* (x'_2 ).
\end{align*}

\item[(iv)] Let $t_0^*$ be the number defined in \eqref{t^*_1}. Then,
\begin{align*}
\lim_{|x'| \rightarrow \infty} t_n^* (x') = t_0^*.
\end{align*}
\end{itemize}

\end{theorem}

\begin{theorem}\label{maintheorem(NS)}
\begin{itemize}

\item[(i)]
Let $\phi$ satisfy the conditions   (i) of Theorem \ref{pressure1103} and   Theorem \ref{theorem260403}.     
Moreover, we assume that   $ \phi \in C^{ 1 +\frac{\al}2} ( 0, T )$  and 
\begin{align}\label{asumption260410} 
\begin{split}
\int_{0}^t   \phi' (s)     (t-s)^{-\frac12}  ds > c t^{\frac12 +\frac{\al}2}, \quad 0 < t< 1,\quad\mbox{and}\quad
 \min_{t\in (1,2)} (\phi(t) + |\phi'(t)|) > \de_1>0.
 \end{split}
\end{align} 
Let  $ c_{*} = \max(c_{*4}, c_{*5} )$, where $c_{*4}$ and $c_{*5}$ be numbers defined in (i) of   Theorem \ref{pressure1103} and (i) of Theorem \ref{theorem260403}, respectively.
Let $c_i$  be the numbers  defined in (i) of   Theorem \ref{pressure1103} and   (i) of  Theorem \ref{theorem260403}.   Assume that $ c_* < c_{**}$.   Then there exists a constant  $ a_0 ( c_{**} )  >0$ such that for every \(a\in(0,a_0)\), the problem \eqref{nse-30} with boundary data \(ag\) admits a solution
$\displaystyle 
u\in C^{2+\alpha,\,1+\frac{\alpha}{2}}(\mathbb{R}_+^n\times(0,T))$
with corresponding pressure \(\pi\), and the pair \((u,\pi)\) satisfies the following property:
If $ c_{*} < |x'| < c_{**}$  with   $ |x'| < c_i x_i, \, 1 \leq i \leq n-1$, then $ D_{x_i} \pi (x',0,t) < 0$ for $1<t <2$ and  there exists $t_i^* \in (1,  2)$ such that $ (x', 0, t_i^*)$ is a point of boundary layer  separation  of $u_i$.
 
\item[(ii)] 
Let $\phi$ satisfies  Assumption \ref{assume3} and $\eqref{asumption260410}_1$. Let $ c_{*2}$ and $c_{2i}$ be numbers defined in (ii) of Theorem \ref{maintheorem}.  Assume that $ c_* < c_{**}$.  Then there exists a constant $ a_0 ( c_{**} )  >0$ such that, for every \(a\in(0,a_0)\), the problem \eqref{nse-30} with boundary data \(ag\) admits a solution $\displaystyle u \in C^{2 +\al, 1 +\frac{\al}2}(\R_+ \times (0, T))$  with corresponding pressure \(\pi\). Moreover, the pair \((u,\pi)\) has the following property:
If $ c_{*} < |x'| < c_{**}$  and  $ |x'| < c_i x_i, \, 1 \leq i \leq n-1$, then the \(i\)-th tangential component \(u_i\) has no boundary layer separation at any boundary point \((x',0,t)\).

\end{itemize}

\end{theorem}

The paper is organized as follows.
Section~\ref{SS-half} introduces the notation used throughout the paper, recalls
the solution formula, and presents a decomposition of the solution.
In Section~\ref{SS-derivative}, we derive decomposition formulas for the normal
derivatives of the tangential velocity components and for the pressure on the
boundary.
Section~\ref{SS-thms} is devoted to the proofs of the main theorems for the
Stokes system, while Section~\ref{NS-thm} provides the proof for the
Navier--Stokes equations. Technical lemmas are proved in the Appendix.

%%%%%%%%%%%%%%%%%%%%%%%%%%%%%%%%%%%%%%%%%%%%%%%%%
%%%%%%%%%%%%%%%%%%%%%%%%%%%%%%%%%%%%%%%%%%%%%%%%%
%%%%%%%%%%%%%%%%%%%%%%%%%%%%%%%%%%%%%%%%%%%%%%%%%

\section{Preliminaries }\label{SS-half}
\setcounter{equation}{0}
We denote points in $\mathbb{R}^{n-1}$ by $x'$ and points in the half-space $\mathbb{R}^n_+$ by $x = (x', x_n)$. For a multi-index $\beta$ and a non-negative integer $m$, partial derivatives are denoted by
$\displaystyle D^\be_x D^m_t = \frac{\partial^{|\be|}}{\partial x^{|\be|}} \frac{\partial^m}{\partial t^m}$ 
for a multi-index $\be$ and a non-negative integer $m$. 
Throughout this paper, we define the space-temporal domains  $Q(T) = \mathbb{R} \times (0, T)$, $Q_+(T) = \mathbb{R}^n_+ \times (0, T)$, and $q(T) = \mathbb{R}^{n-1} \times (0, T)$. Furthermore, $c$, $c_i$, and $d_i$ represent generic positive constants depending only on the dimension $n$, whose values may change from line to line.  In particular, $c_*$ denotes a sufficiently large positive constant depending only on $n$.

We define the anisotropic  Holder continuous  spaces  $C^{ k + \al,  \frac{k}2 + \frac{\al}2} ( Q_+(T) )$ for $0 <\al <1$ and $ k \in \{0, \,  1, \, 2 \}$  by
\begin{align*}
\| f \|_{ C^{\al, \frac{\al}2} ( Q_+(T) )} :  &  = \| f \|_{L^\infty(0, T: C^\al (\R_+))  } + \| f \|_{ L^\infty (\R_+; C^{\frac{\al}2} (0, T)) },\\
\| f \|_{ C^{ 1 + \al,  \frac12 + \frac{\al}2} ( Q_+(T) )} :  &  = \|D_x  f \|_{C^{\al, \frac{\al}2} ( Q_+(T) )} + \| f \|_{ L^\infty (\R_+; C^{\frac12 + \frac{\al}2} (0, T))},\\
\| f \|_{ C^{ 2 + \al,  1+ \frac{\al}2} ( Q_+(T) )} :  &  =  \|   f \|_{C^{1 + \al,  \frac12 + \frac{\al}2} ( Q_+(T) )}  + \|D^2_x  f \|_{C^{\al, \frac{\al}2} ( Q_+(T) ) } + \| D_t f\|_{C^{\al, \frac{\al}2} ( Q_+(T) )}.
\end{align*}
The spaces  $C^{ k + \al,  \frac{k}2 + \frac{\al}2} ( q(T) )$ 
 are defined analogously.

Let $\Ga$ and $N$ be  the $n $-dimensional Gaussian kernel  and  Newtonian kernel, respectively, defined by
\begin{align*}%\label{H-L-10}
\Ga(x,t) = \left\{ \begin{array}{cc}\vspace{2mm}
\frac1{(4\pi t)^{\frac{n}2}} e^{-\frac{ |x|^2}{4t}}& \quad t > 0,\\
0 &\quad t < 0,
\end{array}
\right.
\quad N(x) = \left\{\begin{array}{ll}\vspace{2mm}
-\frac1{n(n -2)\om_n} |x|^{-n +2} &\quad n \geq 3,\\
\frac1{2\pi} \ln |x| &  \quad n =2,
\end{array}
\right.
\end{align*}
where \(\omega_n\) is the volume of the \(n\)-dimensional unit ball. We denote that $\Ga'$  is  $n-1 $-dimensional Gaussian kernel.

% Denote the integral
%\begin{align}
%    L_i(x,t):=\int_0^{x_n}\int_{\Sigma}\partial_{z_n}\Gamma(z,t)\partial_{x_i}N(x-z)dz
%\end{align}
%so that \(L_{ij}=\partial_{x_j}L_i\).
%We define the following spaces
%\begin{enumerate}
%    \item \(\Sigma:=\mathbb{R}^{n-1}\)
%    \item \(\Sigma_{(k)}:=\left\{x_{(k)}'':=(x_1,\cdots,x_{k-1},x_{k+1},\cdots, x_{n-1})\big| x_i\in \mathbb{R}\right\}\) for \(k=1,\cdots, n-1\).
%\end{enumerate}
%We will frequently use the following properties of the kernel \(L_i\):
%\begin{lemm}\label{L_i-properties}
%    Let \(1\leq i \leq n-1\) and \(1\leq j \neq i \leq n-1\).
%    \begin{enumerate}
%        \item \(L_i\) is odd in \(x_i\) and even in \(x_j\).
%        \item \(L_i\) is negative for \(x_i>0\) and is increasing in \(x_j\) for \(x_i, x_j>0\).
%    \end{enumerate}
%\end{lemm}
% The proof of Lemma \ref{L_i-properties} is in Appendix \ref{proofoflemma_L_i-properties}.
 
It is known that the solution  $w$  to the system \eqref{Stokes-10} with boundary data $g$  is represented by
\begin{align}\label{rep-bvp-stokes-w}
w_i(x,t) & = \sum_{j=1}^{n}\int_0^t \int_{\Rn} K_{ij}(
x^{\prime}-y^{\prime},x_n,t-s)g_j(y^{\prime},s) dy^{\prime}ds,
\end{align}
where  the Poisson kernel $K_{ij}$ of the Stokes equations in $Q_+ (\infty)$ is  given by   (see \cite{So}):
\begin{align}\label{Poisson-tensor-K}
\begin{split}
  K_{ij}(x,t) &  =  -2 \delta_{ij} D_{x_n}\Ga(x,t)
+ L_{ij} (x,t)  +2  \de_{jn} \de(t)  D_{x_i} N(x), \quad
i,j=1,2,\cdots, n,
\end{split}
\end{align}
with
\begin{align*}%\label{L-tensor}
L_{ij} (x,t) & = 4  D_{x_j}\int_0^{x_n} \int_{\Rn}   D_{z_n}
\Ga(z,t)  D_{x_i}   N(x-z)  dz.
\end{align*}
The associated pressure $p$ is  given by
\begin{align}\label{representationpressure}
\begin{split}
p(x,t) = & -2  \sum_{j =1}^{n}\int_{\Rn} D_{x_j} D_{x_n} N(x' -y', x_n)    g_j(y', t) dy'   -2  D_t N*'  g_n (x,t) \\
& +    4\sum_{j =1}^{n}
 (D_t-\De') \int_0^t  \int_{\Rn} g_j (z',s) \int_{\Rn} D_{y_j} N(y', x_n) \Ga(x'-y' -z',0,t-s) dy' dz'ds.                                                                                                                                                                                                                                                                                                                                                                                                                                                                                                                                                                                                                                                                                                                                                                                                                                                                                                                                                                                                                                                                                                                                                                                                                                                                        
\end{split}
\end{align}
We also recall (see \cite{So})
\begin{align}\label{1006-3}
L_{ij} = L_{ji},\quad 1 \leq i,j \leq n-1,  \qquad
  L_{in} =
L_{ni}  + B_{in},\quad  i \neq n,
\end{align}
%where $ B_{in}$ defined as
\begin{equation*}%\label{B-tensor}
B_{in}(x,t) := 4p. v.\int_{\Rn} D_{x_n} \Ga(x^\prime -z^\prime ,
x_n, t)  D_{z_i} N( z^{\prime},0) dz^\prime =2R'_i (D_{x_n}\Gamma (\cdot, x_n, t))(x'),
\end{equation*}
where $ R' = (R'_1, \cdots, R'_{n-1} )$ denotes the $n-1$-dimensional Riesz transform. 
Based on     \eqref{rep-bvp-stokes-w}, \eqref{Poisson-tensor-K}  and  $\eqref{1006-3}_2$, the solution  $w = (w_1, \cdots, w_n)$ is represented as follows:
\begin{align}\label{repre0926}
\begin{split}
w_i(x,t) &  = w_i^G (x,t) +  \sum_{1 \leq j \leq n-1}  w_{ij}(x,t)  + w^{L}_{i} (x,t) + w_i^B(x,t) + w_i^N(x,t),\quad 1 \leq i \leq n-1.
\end{split}
\end{align}
Here  the components in \eqref{repre0926} are defined by
\begin{align}\label{260409-1}
\begin{split}
w_i^G(x,t)& = -2 \int_0^t \int_{\Rn} D_{x_n} \Ga(x' -y', x_n, t-s) g_i(y',s) dy'ds,\\
w_{ij} (x,t) & = \int_0^t \int_{\Rn} L_{ij} (x'-y', x_n, t-s) g_j(y',s)  dy' ds, \quad  1 \leq  i, \, j \leq n,\\
 w^{L}_{i} (x,t) & =  \int_0^t \int_{\Rn} L_{ni} (x'-y', x_n, t-s) g_n(y',s)  dy' ds,\\
% w^L_{jj}(x,t) & =   \int_0^t \int_{\Rn} L_{jj} (x'-y', x_n, t-s) g_n(y',s)  dy' ds,\\
w^B_i (x,t) &  =\int_0^t \int_{\Rn}  B_{in}(x' -y', x_n, t-s) g_n(y',s) dy'ds, \\
w^N_i (x,t)  &= 2   \int_{\Rn} D_{x_i} N(x'-y', x_n)  g_n (y',t) dy'.
\end{split}
\end{align}
Finally, for    $1 \leq i \leq n-1$, we note that  
\begin{align}\label{NB0926}
\begin{split}
w^N_i (x,t)  &=  2\int_{\Rn} D_{x_n} N(x'-y', x_n) R'_i g_n (y',t) dy', \\
w^B_i (x,t) &  =2 \int_0^t \int_{\Rn} D_{x_n} \Ga(x' -y', x_n, t-s) R'_ig_n(y',s) dy'ds.
\end{split}
\end{align}

%%%%%%%%%%%%%%%%%%%%%%%%%%%%%%%%%%%%%%%%
%%%%%%%%%%%%%%%%%%%%%%%%%%%%%%%%%%%%%%%%
%%%%%%%%%%%%%%%%%%%%%%%%%%%%%%%%%%%%%%%%

\section{Estimates of  $  D_{x_n} w_i $,  $ i \neq n  $  and $p$ at the boundary}\label{SS-derivative}
\setcounter{equation}{0}

Next proosition is prepared to the expression of $  D_{x_n} w_i $,  $ i \neq n  $  and $p$ at the boundary.

 \begin{lemma}\label{thoe1021-1} 
Let $(w,p)$ be the  solution to system \eqref{Stokes-10} with boundary data $g$.   For $i \neq n$,   
 \begin{align}\label{normal}
\begin{split}
D_{x_n }w_i (x',0,t) &  = \frac{1}{2 \pi^ \frac12}\int_0^t  \frac1{(t -s)^\frac32}    \int_{\Rn}   \Ga' (x' -y', t-s) \big(  g_i(y',s)   -   g_i(x',s)    \big)  dy' ds\\
&\quad +\frac{1}{ 2\pi^ \frac12}  \int_{-\infty}^t     \frac1{(t -s)^\frac32}   \int_{\Rn}  \Ga' (x' -y', t-s) 
\big(  g_i(x',s)   -  g_i(x',t)    \big)    dy' ds\\
&\quad   -2  \sum_{1 \leq j \leq n-1}  D_{x_i}R'_j  g_j (x',t) \\
& \quad -\frac1{2\pi^\frac12}\int_0^t  \frac1{(t -s)^\frac32}    \int_{\Rn}   \Ga' (x' -y', t-s) \big( R'_i g_n(y',s)   - R'_i g_n(x',s)    \big)  dy' ds\\
&\quad -  \frac1{2\pi^\frac12} \int_{-\infty}^t     \frac1{(t -s)^\frac32}   \int_{\Rn}  \Ga' (x' -y', t-s) 
\big( R'_i g_n(x',s)   - R'_i g_n(x',t)    \big)    dy' ds.
\end{split}
\end{align}
\begin{align}\label{representationpressure-1}
\begin{split}
p(x,t) 
& = -4 \sum_{j =1}^{n} D_{x_j} D_{x_n} \int_{\Rn} N(x' -y', x_n) g_j(y', t) dy' - 2 D_t N*' g_n (x,t) \\
& \quad - 2\sum_{j =1}^{n} \int_{-\infty}^t \int_{\Rn}\frac{g_j (z',s) -g_j(z',t)}{t -s} 
  \int_{\Rn} D_{y_j} N(y', x_n) \Ga(x'-y' -z',0,t-s) dy' dz' ds. 
\end{split}
\end{align}
\end{lemma}
 \begin{proof}
The proof of Lemma \ref{thoe1021-1} is provided   in  Subsection \ref{prooflemma1} of Appendix.
 \end{proof}
 
In case that $g = g_n {\bf e}_n$ with \eqref{0502-6}, from \eqref{normal}, we have
 \begin{align} \label{shortg}
D_{x_n} w_i (x', 0,t) & =  - \frac1{\pi^\frac12}  \int_0^t  \frac{ \phi(s)}{(t-s)^\frac32}     f_i (x', t-s) ds    + \frac1{\pi^\frac12}  R'_i \psi(x) \int_{-\infty}^t     \frac{\phi(t) - \phi(s)}{(t-s)^\frac32}  ds, \quad  1 \leq i \leq n-1, 
\end{align}
where 
\begin{align}\label{defoff}
  f_i (x', t-s)  = \int_{\Rn}   \Ga'(x' -y', t-s) \big( R'_i \psi(y')   - R'_i \psi(x')    \big)  dy'.
\end{align}
 
Similary it follows from \eqref{representationpressure-1} that
\begin{align}\label{260321-1}
\begin{split}
p(x,t) 
& = -4 D_{x_n} D_{x_n}\int_{\Rn} N(x' -y', x_n) g_n (y', t) dy' - 2 D_t N*' g_n (x,t) \\
& \quad - \frac{1}{\sqrt{ \pi}} \int_{-\infty}^t \int_{\Rn}\frac{g_n (z',s) -g_n(z',t)}{(t -s)^{\frac{3}{2}}} 
  \int_{\Rn} D_{x_n} N(y', x_n) \Ga'(x'-y' -z',t-s) dy' dz' ds.
\end{split}
\end{align}

In the next lemma, we estimate the tangential derivatives of $f_i$ in \eqref{defoff}.

\begin{lemma}\label{lemma1116-1}
Let $ f_i$ , $i\neq n$  be the function defined in \eqref{defoff} and $ \be = (\be_1,\be_2, \cdots, \be_{n-1})   \in ( {\mathbb N} \cup \{ 0\} )^{n-1}$ with $ |\be| = \sum_{1 \leq k \leq n-1} \be_k$.
Then, there exists  $ c_* > 2$ such that for $ 0 < t <2$ and $ |x'| \geq c_*$, 
$D_{x'}^\be  f_i$ can be  decomposed as $ D_{x'}^\be  f_i(x',t) = F_1(x', t) + F_2(x', t)$, where 
\begin{align*}
F_1(x',t) &  =   \De'  D_{x'}^\be   R'_i \psi (x'  )    t, \qquad |F_2 (x',t)|   \leq c t^\frac32 |x'|^{-n -2 -|\be|}.
\end{align*}

\end{lemma}

\begin{proof}
 The proof of Lemma \ref{lemma1116-1} is provided in Appendix \ref{prooflemma3-1}.
 \end{proof}

Next, we consider $\De' D^\be_{x'}     R_i'  \psi (x') $ to extract the main term for $f_i$ and its derivatives. Note that   for $|\be| \geq 0$, we have 
\begin{align}\label{260404-1}
\De' D^\be_{x'}     R_i'  \psi (x') &  =\int_{\Rn} \Big( \De'  D^\be_{x'}  D_{x_i}    N(x'-z', 0) -\De'  D^\be_{x'}  D_{x_i}  N(x', 0)\Big) \psi(z')  + \De'   D^\be_{x'}   D_{x_i} N(x', 0) c_\psi,
\end{align}
where $ c_\psi = \int_{\Rn} \psi(x') dx'$. 
 Since $ supp\, (\psi) \subset B'(1)$ and $2< |x'|$, the first term in \eqref{260404-1} is estimated as follows:
\begin{align}\label{260404-2}
\begin{split}
 &\abs{\int_{\Rn} \Big( \De'  D^\be_{x'}   D_{x_i}   N(x'-z', 0) -\De' D_{x'}^\be  D_{x_i}  N(x', 0)\Big) \psi(z') dz'  }\\
 \leq    &\int_{\Rn}   \abs{\na' \De'  D^\be_{x'}   D_{x_i}  N(x'-\xi', 0)  \cdot z'} \psi(z') dz'
  \leq  c |x'|^{-n -2-|\be|}.
 \end{split}
\end{align}
It is direct that
\begin{align}\label{260403-2}
\begin{split}
   D_{x_i} \De' N(x',0)   =  \frac{1}{\om_n} \frac{x_i } {|x'|^{n+2}},\qquad 
       D^2 _{x_i} \De'   N(x', 0)    =  \frac{1}{\om_n}    \frac{1} {|x'|^{n+2}} - (n+2) \frac{x_i^2} {|x'|^{n+4}},\\
      \sum_{1 \leq k \leq n-1} \De'  D^2_{x_k} N(x', 0)    = - \frac{3}{\om_n}  \frac1{|x'|^{n+2}},\qquad 
           \sum_{1 \leq k \leq n-1}  D_{x_i} \De'  D^2_{x_k} N(x', 0)    =  \frac{3(n+2)}{\om_n}  \frac{x_i }{|x'|^{n+4}}.
           \end{split}
   \end{align}

Let $ \ep_0 > 0$ be a given small positive number.    From Lemma \ref{lemma1116-1}, \eqref{260404-1}, \eqref{260404-2} and $\eqref{260403-2}$,  there exists $ 2< c_*(\ep_0) $ such that if  $ c_*(\ep_0) <|x'|$, then  we have 
\begin{align}\label{260323-5}
 \begin{array}{c}
%  c_n \int_0^t  \frac{ \phi(s)}{(t-s)^\frac32}   \sum_{1 \leq k \leq n-1}  D_{x_k} f_k (x', t-s) ds   &  \approx   -c_\psi   |x'|^{-n-2}   \int_0^t   (t -s)^{-\frac12}    \phi( s)  ds.\\
 \vspace{4mm}          (  -\frac{3 }{\om_n} - \ep_0) |x'|^{-n-2} c_\psi  t\leq    \sum_{1 \leq k\leq n-1}   D_{x_k}  f_k  (x'  ,t)   
 \leq  (   -\frac{3 }{\om_n} + \ep_0) |x'|^{-n -2}   c_\psi t,  \\
\vspace{4mm}   (\frac1{\om_n} - \ep_0) x_i^{-n -1}t c_\psi \leq f_i (x_i {\bf e}_i,t)   \leq   (\frac1{\om_n} + \ep_0) x_i^{-n -1} c_\psi t,\\
\vspace{4mm}
  ( - \frac{n +1}{\om_n} - \ep_0) x_i^{-n -2} c_\psi  t\leq D_{x_i} f_i  (x_i {\bf e}_i,t)   \leq   (-   \frac{n+1}{\om_n} + \ep_0) x_i^{-n -2} c_\psi t,\\
       (  \frac{3(n+2)}{\om_n} - \ep_0) x_i^{-n -3} c_\psi  t\leq    \sum_{1 \leq k \leq n-1}  D_{x_i} D_{x_k}  f_k  (x_i {\bf e}_i ,t)   \leq   (   \frac{3(n+2)}{\om_n} + \ep_0) x_i^{-n-3}  c_\psi t.
 \end{array}
\end{align}

%%%%%%%%%%%%%%%%%%%%%%%%%%%%%%%%%%%%%%%%
%%%%%%%%%%%%%%%%%%%%%%%%%%%%%%%%%%%%%%%%
%%%%%%%%%%%%%%%%%%%%%%%%%%%%%%%%%%%%%%%%

\section{Proofs of Theorems for the Stokes system}\label{SS-thms}

  \subsection{Proof of Theorem \ref{pressure1103}} %(Adverse Pressure Gradient)}
  \setcounter{equation}{0}
  \label{mainsection}

We note first that 
\[
\int_{\Rn} D_{x_n} N(y', x_n) \Ga' (x'-y' -z', t-s) dy' \Big|_{x_n =0} = \frac12 \Ga' (x'-z', t-s).
\]
Combining it with \eqref{260321-1}, we have 
\begin{align}\label{pressure-6-10}
\begin{split}
p(x',0,t) 
& = 4 \De'\int_{\Rn} N(x' -y', 0) g_n (y', t) dy' - 2 \int_{\Rn} N(x' -y', 0)  D_t g_n (y', t) dy'    \\
& \quad - \frac{1}{2\sqrt{ \pi}} \int_{-\infty}^t \int_{\Rn}\frac{g_n (z',s) -g_n(z',t)}{(t -s)^{\frac{3}{2}}} 
    \Ga'(x' -z',t-s)  dz' ds.
\end{split}
\end{align}
It follows from  direct calculation that
\begin{align}\label{260314-1}
\begin{split}
\frac{1}{2\sqrt{ \pi}} \int_{-\infty}^t \frac{1}{(t -s)^{\frac{3}{2}}} \Ga' (x', t-s) ds 
& = |x'|^{-n} \pi^{-\frac{n}2} \int_0^\infty s^{\frac{n}2 -1} e^{-s} ds = - 2\De' N(x', 0).
\end{split}
\end{align}
Since   $2 < |x'| < c_1 x_i$ and ${\rm supp} \, g_n(\cdot, t) \subset B'(1)$, from $\eqref{260314-1}$, we have 
\begin{align*}
\begin{split}
& \frac{1}{2\sqrt{\pi}} \int_{-\infty}^t \int_{\Rn} \frac{g_n (z', t)}{(t -s)^{\frac{3}{2}}} 
   \Ga'(x' -z',t-s)   dz' ds = (-2\De') \int_{\Rn} g_n (z', t)   N(x'  -z',0)   dz'.
\end{split}
 \end{align*}
Then, using the above idenity, it follows from \eqref{pressure-6-10} that
 \begin{align*}
\begin{split}
p(x',0,t) 
& = 2\De'\int_{\Rn} N(x' -y', 0) g_n (y', t) dy' - 2 \int_{\Rn} N(x' -y', 0)  D_t g_n (y', t) dy'    \\
& \quad - \frac{1}{2\sqrt{\pi}} \int_{-\infty}^t \int_{\Rn}\frac{g_n (z',s)  }{(t -s)^{\frac{3}{2}}} 
    \Ga'(x' -z',t-s)  dz' ds.
\end{split}
\end{align*}
Taking derivative in $x_i$,  $ 1 \leq i \leq n-1$, we get
  \begin{align*}
\begin{split}
D_{x_i} p(x',0,t) 
& = 2D_{x_i} \De'\int_{\Rn} N(x' -y', 0) g_n (y', t) dy' - 2 D_{x_i} \int_{\Rn} N(x' -y', 0)  D_t g_n (y', t) dy'    \\
& \quad - \frac{1}{\sqrt{4 \pi}} \int_{-\infty}^t \int_{\Rn}\frac{g_n (z',s)  }{(t -s)^{\frac{3}{2}}} 
    D_{x_i}\Ga'(x' -z',t-s)  dz' ds.
\end{split}
\end{align*}
For $2 < |x'| < c_1 x_i$, since   $2 < |x'| < c_1 x_i$ and ${\rm supp} \, g_n(\cdot, t) \subset B'(1)$, we obtain
\begin{align}\label{0223-6}
\begin{split}
& 2D_{x_i} \De'\int_{\Rn} N(x' -y', 0) g_n (y', t) dy'  = \phi(t) \frac1{\om_n} \int_{\Rn} \frac{x_i -y_i}{|x' -y'|^{n+2}} \psi(y') dy' \approx  \phi(t)|x'|^{-n -1},\\
& 2 D_{x_i} \int_{\Rn} N(x' -y', 0)  D_t g_n (y', t) dy'  = \frac12  \phi' (t) R'_i \psi(x') \approx \phi' (t) |x'|^{-n +1},\\
 & - \frac{1}{\sqrt{4 \pi}} \int_{-\infty}^0 \frac{\phi( s)  }{(t -s)^{\frac{3}{2}}} 
  \int_{\Rn} D_{x_i} \Ga' (x' -z', t-s) \psi(z') dz' ds  \approx \int_0^t \frac{\phi (s) x_i}{(t-s)^{\frac{n}{2} +2}} e^{-\frac{|x'|^2}{t-s}} ds.
\end{split}
\end{align}
Consequently, the gradient of the pressure satisfies:
\begin{align*}
D_{x_i} p(x', 0,t) \approx \phi(t)|x'|^{-n -1}  - \phi' (t) |x'|^{-n +1} + \int_0^t \frac{\phi (s) x_i}{(t-s)^{\frac{n}{2} +2}} e^{-\frac{|x'|^2}{t-s}} ds.
\end{align*}
Assuming $\phi$ is increasing in $(0, 1)$, for $0 < t < 1$, the integral term is bounded as follows:
\begin{align*}
\int_0^t \frac{\phi (s) x_i}{(t-s)^{\frac{n}{2} +2}} e^{-\frac{|x'|^2}{t-s}} ds &\leq c  \phi(t) |x'| \int_0^t \frac{1}{s^{\frac{n}{2} +2}} e^{-\frac{|x'|^2}{ s}} ds
 = c  \phi(t) |x'|^{-n -1} \int_{\frac{|x'|^2}t}^\infty s^{\frac{n}2}    e^{-s} ds\\
& \leq  c  \phi(t) |x'|^{-1}  t^{-\frac{n}2} e^{-\frac{|x'|^2}t}    \leq  c  \phi(t) |x'|^{-n-2}  t^{\frac{1}2}.
\end{align*}
Thus, for $0 < t < 1$, the leading order terms of $D_{x_i} p$ is given as follows:
\begin{align}\label{1031-1}
D_{x_i} p(x', 0,t) \approx \phi(t)|x'|^{-n -1}  - \phi' (t) |x'|^{-n +1}.
\end{align}
If $\phi(t) \leq c t \phi'(t)$ for $0 < t < t_1$, it follows from \eqref{1031-1} that:
  \begin{align*} 
  D_{x_i} p(x', 0,t)& <  c_1 \phi (t)  |x'|^{-n -1}   - c_2 t^{-1} \phi(t)  |x'|^{-n +1}   = |x'|^{-n +1} \phi(t) \big( c_1 |x'|^{-2}-  c_2    t^{-1}   \big). \label{eq_proof1}
  \end{align*}
  For $x'$ satisfying $t < t_1 < \frac{c_2}{c_1} |x'|^{2}$, we have $D_{x_i} p(x', 0,t) < 0$. This  implies that $t_1 < t_{p,i}^* (x')$, which completes the proof of part (i). 

Conversely, if $\phi'(t_2) = 0$ for some $t_2 \leq 1$, then:
  \begin{align*}
  D_{x_i} p(x', 0,t_2)& > - c_3 |x'|^{-n +1} \phi' (t_2) + c_4 |x'|^{-n -1} \phi(t_2)  = c_4 |x'|^{-n -1} \phi(t_2) > 0,
  \end{align*}
which implies $t_{p,i}^* (x') \leq t_2$. This completes the proof of part (ii).

Noting that   $ - \frac{1}{\sqrt{4 \pi}} \int_{-\infty}^0 \frac{\phi( s)  }{(t -s)^{\frac{3}{2}}} 
  \int_{\Rn} D_{x_i} \Ga' (x' -z', t-s) \psi(z') dz' ds > 0$ and $\phi'(t) \leq 0$ for $ 1 < t<2$, we can see that  there are $ c_1, c_2 > 0$ such that 
 \begin{equation}  \label{pressure260410} 
D_{x_i} p(x', 0,t) \geq  c_1 \phi(t)|x'|^{-n -1}  -  c_2\phi' (t) |x'|^{-n +1} >0.
\end{equation}
This establishes the assertion (iii), and thus we completes the proof of Theorem \ref{pressure1103}.
\qed

\subsection{Proof of Theorem \ref{maintheorem} }
\label{mainsection}

Suppose that $\phi$ satisfies Assumption \ref{assume2}. 
Recall that  $\De' R'_i \psi(x') \approx c_\psi |x'|^{-n -1}$ for large $|x'|$ satisfying $ |x'| \leq c_i x_i$. Using Lemma \ref{lemma1116-1} and \eqref{260323-5},  we have for  large $|x'|$ satisfying $ |x'| \leq c_i x_i$.
\begin{align}\label{0106-2}
\begin{split}
 - c_n \int_0^t  \frac{ \phi(s)}{(t-s)^\frac32}     f_i (x', t-s) ds   % & \leq c     |x'|^{-n-1} \int_0^t  \frac{x_n}{(t -s)^\frac32}e^{-\frac{x_n^2}{4(t-s)}}  (t -s)   \phi(s) ds\\
&  \approx   -c_\psi   |x'|^{-n-1}   \int_0^t   (t -s)^{-\frac12}    \phi( s)  ds.
\end{split}
\end{align}
Since $\phi$ is increasing in $(0,1)$, we oobserve  that for $ 0 < t < 1$
\begin{align}\label{260322-1}
\int_0^t   (t -s)^{-\frac12}    \phi( s)  ds \leq 2t^\frac12 \phi(t)  \leq 2  \phi(t), \qquad 
 M(t) \geq  \int_{-\infty}^0     \frac{\phi(t) }{(t-s)^\frac32}  ds = 2\phi(t) t^{-\frac12} \geq 2\phi(t).
 \end{align}
Therefore, it follows  from \eqref{shortg} that for $ 0 < t< 1$ 
\begin{align}\label{260319-1}
D_{x_n} w_i (x',0,t) \geq    - c_1 c_\psi \phi(t)      |x'|^{-n-1} + c_2  c_\psi\phi(t)    |x'|^{-n +1}.
\end{align}
We denote for convenience $\displaystyle \Psi(t) = \int_0^t   (t -s)^{-\frac12}    \phi( s)  ds$.
On the other hand, for $ t_0< t < t_1 \leq 2$ we obtain from   Assumption \ref{assume2} and \eqref{0106-2}
 \begin{align}\label{260319-2}
D_{x_n} w_i (x',0,t) \leq    - c_3 c_\psi \Psi(t)    |x'|^{-n-1} + c_4 |x'|^{-n +1} M(t) \leq     - c_5 |x'|^{-n +1}.%  -c_4 c_0  c_\psi |x'|^{-n +1}.
\end{align}
Combinig \eqref{260319-1} and \eqref{260319-2}, there exists $c_*> 1$ and $ c_i > 0$  such that if $ c_* < |x'| < c_i x_i$, then 
\[
D_{x_n} w_i(x', 0,t) > 0\,\,\text{ for } 0 < t< 1,\qquad D_{x_n} w_i (x',0,t) < 0\,\,\text{ for } t_0 < t<t_1.
\]
%$ D_{x_n} w_i(x', 0,t) > 0$ for $ 0 < t< 1$ and $ D_{x_n} w_i (x',0,t) < 0$ for $ t_0 < t<t_1$.
Now, fix $x'$ such that $ c_* < |x'| < c_i x_i$.  
We define $t_i^*(x')$ as  
\begin{align*}
t^*_i (x')   = \sup \{ r \, | \, D_{x_n} w_i (x', 0,t) \geq 0 \quad \mbox{for all} \quad  0 < t< r \}.
\end{align*}
By the continuiety of $ D_{x_n} w_i$, we have $ D_{x_n} w_i (x', 0,t) \geq 0$ for all $ 0\leq t \leq t_i^*(x')$. Moreover, since $ f_i > 0$,  we get  $ t_i^*(x') < t_0$, where $ 1 <t_0$ is defiend in Assumption \ref{assume2}.

Since $D_{x_n} w_i (x', 0,t)$ is continuous in $\Rn \times (0, T)$, $D_{x_n} w_i (x', 0, t) > 0$ implies that 
$(x', t)$ satisfies condition ${\mathcal P}_+$. From (3) of  Theorem \ref{pressure1103},  for $ 1 < t\leq t_i^*(x')$, we have 
\begin{align}\label{260324-1}
D^2_{x_n}w_i (x', 0, t) = D_{x_i} p(x', 0,t ) > 0.
\end{align}

Since $w_i (x', 0, t ) = 0$, if  $D_{x_n} w_i (x', 0, t ) = 0$ for $ 0 < t \leq t_i^*(x')$, then from \eqref{260324-1},  we have 
\begin{align*}%\label{1031-2}
w_i(x', x_n, t) \approx D_{x_i} p(x', 0,t) x_n^2 > 0
\end{align*}
for  small $x_n$. Hence,  $(x', t)$ satisfies the condition $ {\mathcal P}_+$ for all $ 0< t\leq t_i^*(x')$.  Therefore, we conlude that $(x',t^*(x') )$ is boundary layer separation of $w_i$. This completes the proof  (i) of Theorem \ref{maintheorem}.

Next, suppose  $\phi$ satisfies Assumption \ref{assume3}.  From \eqref{shortg} and \eqref{0106-2},  for large $|x'|$, we have 
\begin{align}\label{060410-5}
D_{x_n} w_i (x', 0,t)  \approx  -c_\psi   \Psi(t) |x'|^{-n-1}   + |x'|^{-n +1} M(t) > -c_\psi   \Psi(t) |x'|^{-n-1}   + c_0 |x'|^{-n +1} > 0
\end{align}
for all $0 < t < 2$.  Consequently, $(x', t)$ satisfies condition ${\mathcal P}_+$ for all $0 < t < 2$. Hence, $w_i$ exhibits no 
boundary layer separation. This completes the proof of (ii) of Theorem \ref{maintheorem}.
\qed

\subsection{Proof of Theorem \ref{theorem1221} }
\label{mainsection2}

Fix $x_1 > c_*$. In this section, we denote $t^*_i = t_i^* (x_1 {\bf e}_i)$.   For sufficiently large $x_1$, we obtain 
\begin{align}\label{260111-2}
\begin{split}
  \frac1{n\om_n }  \big(1 -\ep_0\big)  c_\psi   x_1^{-n +1}  & \leq \frac1{n\om_n }   c_\psi  \frac{x_1 -1}{(x_1 +1)^n}  \leq  R'_i \psi(x_1 {\bf e}_i)  < \infty,\\
-  \frac{n-1}{n\om_n }  \big( 1  +\ep_0   \big) c_\psi   x_1^{-n} &  \leq  \frac{1}{n\om_n }  c_\psi   \frac{(x_1 -1)^2 -n(x_1 + 1)^2}{(x_1 -1)^{n+2}}  \leq  D_{x_i} R'_i \psi (x_1 {\bf e}_i) <0.
\end{split}
\end{align}

From $ \eqref{260323-5}_2$ and $ \eqref{260323-5}_3$, it follows that:
\begin{align}\label{260111-3}
\begin{split}
f_i (t^*_i -s, x_1 {\bf e}_i) & \leq   \frac1{\om_n} \big( 1  + \ep_0\big) c_\psi   x_1^{-n-1} (t_i^* -s)  < \infty,\\ 
-\infty< D_{x_i} f_i (t_i^* -s, x_1 {\bf e}_i) 
% & \leq -  \frac{n(n+1)}{n-1} x_1^{-n-2} (t^* (x_1{\bf e}_i) -s) 
%   + c x_1^{-n -3} (t^* (x_1{\bf e}_i) -s)^{\frac{3}{2}}\\
& \leq -      \frac{n+1}{\om_n} \big( 1 - \ep_0 \big)   c_\psi x_1^{-n-2} (t_i^* -s). 
\end{split}
\end{align}

By setting $\epsilon_0 = (4n)^{-1}$ and combining \eqref{260111-2} and \eqref{260111-3}, we derive:
\begin{align}\label{260307-1}
\begin{split}
& \frac{D_{x_i} R'_i \psi(x_1 {\bf e}_i)}{R'_i \psi (x_1 {\bf e}_i)} f_i (t_i^* -s, x_1 {\bf e}_i) - D_{x_i} f_i (t_i^* -s, x_1 {\bf e}_i) \\
%& \geq \Big( \frac{ \Big( -(n-2)(n-1) +\ep \Big)}{(n-2)-\ep} c_0 \big( \frac{n}{n-1} +\ep \big) 
% + c_0 \big( \frac{n(n+1)}{n-1} -\ep\big) \Big) x_1^{-n-2} (t^* (x_1{\bf e}_i) -s)\\
& \geq  -  \frac{  \frac{n-1}{n \om_n }  \big( 1  +\ep_0   \big)  }{  \frac1{n \om_n}  \big(1 -\ep_0\big) }    \frac1{\om_n} \big( 1  + \ep_0\big) c_\psi   x_1^{-n-2} (t_i^* -s)
 +    \frac{n +1}{\om_n} \big( 1 - \ep_0 \big)   c_\psi x_1^{-n-2} (t_i^* -s)\\
 &  =  \big(  -\frac{(1 +\ep_0)^2}{1-\ep_0}  (n-1) + (1-\ep_0)(n+1) \big)  \frac1{\om_n}  c_\psi         x_1^{-n-2} (t_i^* -s)\\
  &  = \frac1{\om_n} \frac{8n^2 +1}{ 2n (4n-1)}   c_\psi        x_1^{-n-2} (t_i^* -s).
\end{split}
\end{align}
We note via $D_{x_n} w_i (x_1 {\bf e}_i, 0, t)=0$ in \eqref{shortg}  that
\begin{align*} 
\begin{split}
\int_{-\infty}^{t_i^*} \frac{\phi(t_i^*) -\phi(s)}{(t_i^* -s)^{\frac{3}{2}}} ds 
= \int_0^{t_i^*} \frac{\phi(s)}{(t_i^* -s)^{\frac{3}{2}}} \frac{f_i (t_i^* -s, x_1 {\bf e}_i)}{R'_i \psi (x_1 {\bf e}_i)} ds,
\end{split}
\end{align*}
Therefore,  due to \eqref{260307-1}
we obtain 
\begin{align}\label{260111-1}
\begin{split}
D_{x_i} D_{x_n} w_i \Big|_{(x', x_n, t) = (x_1 {\bf e}_i, 0, t_i^*)} 
& = \int_0^{t^*} \frac{\phi(s)}{(t_i^* -s)^{\frac{3}{2}}} 
\Big( \frac{D_{x_i} R'_i \psi(x_1 {\bf e}_i)}{R'_i \psi (x_1 {\bf e}_i)} f_i (t_i^* -s, x_1 {\bf e}_i) 
- D_{x_i} f_i (t_i^* -s, x_1 {\bf e}_i) \Big) ds\\
%\label{260321-5-1-1}
%D_{x_i} D_{x_n} w_i \Big|_{(x', x_n, t) = (x_1 {\bf e}_i, 0, t_i^*)} 
%& \geq c_1 x_1^{-n -2} \int_0^{t^* (x_1{\bf e}_i)} \frac{\phi(s)}{(t^* (x_1{\bf e}_i) -s)^{\frac{1}{2}}} ds\\
&\geq  \frac1{\om_n}  c_\psi   \frac{8n^2 +1}{ 2n (4n-1)}   x_1^{-n-2} \int_0^{t_i^*} \frac{\phi(s)}{(t_i^* -s)^{\frac{1}{2}}} ds 
\geq c_1  x_1^{-n-2}.
\end{split}
\end{align}

Since  $ D_{x_i} D_{x_n} w_i (x', 0,t)$ is uniformly continuous on $B'(2x_1) \times (0,2)$, there exists $\de= \de(x_1)> 0$ such that  if   $ |x_1  -x_i| < \de $, then, from \eqref{260111-1}, this  implies  that
\begin{align*}
D_{x_i} D_{x_n} w_i  ((x_i{\bf e}_i, 0 , t_i^*) \geq  \frac12 c_1  x_1^{-n-2}.
\end{align*}
This ensures that for $x_1 - \delta < y_i < x_1 < x_i < x_1 + \delta$:
\begin{align}\label{260224-5-1}
D_{x_n} w_i (y_i \mathbf{e}_i, 0, t^*_i ) < D_{x_n} w_i (x_1 \mathbf{e}_i, 0, t_i^*) = 0 < D_{x_n} w_i (x_i \mathbf{e}_i, 0, t_i^*).
\end{align}
The left inequality of \eqref{260224-5-1} leads to  
\begin{align*}%\label{0223-1}
t_i^*(y_i {\bf e}_i) < t_i^*(x_1 {\bf e}_i)  \quad \text{for all } x_1 -\de < y_i < x_1.
\end{align*}

Let  $ x_1 < z_1< x_1 +\de $.  From the process to obtain   \eqref{260111-1}, we have 
\begin{align*}%\label{260321-5}
\begin{split}
D_{x_i} D_{x_n} w_i \Big|_{(x', x_n, t) = (z_1 {\bf e}_i, 0, t_i^*(z_1))}  \geq c_1  z_1^{-n-2} \geq c_1   x_1^{-n-2}.
\end{split}
\end{align*}
This implies 
\begin{align*}%\label{260323-7}
 t_i^*(x_1) <   t_i^* (z_1) \quad \mbox{for} \quad x_1 < z_1< x_1 +\de.
\end{align*}
 
 Let $\delta_0 = \sup\{ \delta > 0 \mid t_i^*(x_1 {\bf e}_i) < t_i^* (x_i {\bf e}_i) \text{ for all } x_i \in (x_1, x_1 + \delta) \}$. We have shown $\delta_0 > 0$. Suppose, for the sake of contradiction, that $\delta_0 < \infty$. Applying the same local monotonicity argument at the point $(x_1 + \delta_0) {\bf e}_i$, there exists $\delta_1 \in (0, \delta_0)$ such that for $x_1 +\delta_0 -\delta_1 < y_i < x_1 +\delta_0 < x_i < x_1 +\delta_0 + \delta_1$, we have:
\begin{equation*} %\label{0224-6}
t^*_i (y_i {\bf e}_i) < t^*_i((x_1 +\delta_0) {\bf e}_i) < t^*_i (x_i {\bf e}_i).
\end{equation*}
Then, for all $x_1 < x_i < x_1 + \delta_0 + \delta_1$, it holds that $t^*_i (x_1 {\bf e}_i) < t^*_i (y_i {\bf e}_i) < t^*_i((x_1 +\delta_0) {\bf e}_i) < t^*_i (x_i {\bf e}_i)$. This contradicts the definition of $\delta_0$ as the supremum. Hence, $\delta_0 = \infty$, completing the proof of part (i).

 From Lemma \ref{lemma1116-1}, we have $f, R'_i \psi > 0$ for $x_i > c_*$, 
which implies $t_i^* (x_i {\bf e}_i) < t^*_0 $,  where $t^*_0  $ is defined in \eqref{t^*_1}.  Since $ t^*_i(x_i{\bf e}_i)$ is increasing with respect to $x_i$ and bounded  above , there exists  $r^*  \leq t^*_0$ such that $\displaystyle \lim_{x_i \rightarrow \infty}  t_i^*(x_i{\bf e}_i) =  r^*$. Using   $D_{x_n} w_i (x_i {\bf e}_i, 0, t_i^* (x_i {\bf e}_i)) = 0$, we have 
\begin{align*}
\int_{-\infty}^{t^*_i  (x_i {\bf e}_i)} \frac{\phi(t^*_i (x_i {\bf e}_i)) -\phi(s)}{(t^*_i (x_i {\bf e}_i) -s)^{\frac{3}{2}}} ds 
& = \frac{1}{R'_i \psi (x_i {\bf e}_i)} \int_0^{t^*_i (x_i {\bf e}_i)} \frac{\phi(s)}{(t^*_i (x_i {\bf e}_i) -s)^{\frac{3}{2}}} f(t^*_i (x_i {\bf e}_i) -s, x_i {\bf e}_i) ds.
\end{align*}
By the estimate $\frac{f(x_i {\bf e}_i , t^*(x_i {\bf e}_i ) )}{R'_i \psi (x_i {\bf e}_i )} \leq c x_i^{-2} (t^*_i (x_i {\bf e}_i ) -s)$ from Lemma \ref{lemma1116-1}, taking the limit as $x_i \rightarrow \infty$ yields:
 \begin{align*}
\lim_{x_i \rightarrow \infty} \int_{-\infty}^{t^*_i (x_i {\bf e}_i)} \frac{\phi(t^*_i (x_i {\bf e}_i)) -\phi(s)}{(t^*_i (x_i {\bf e}_i) -s)^{\frac{3}{2}}} ds
= \int_{-\infty}^{r^*  } \frac{\phi(r^*  -\phi(s)}{(r^* ) -s)^{\frac{3}{2}}} ds & = 0.
\end{align*}
This implies $r^* = t^*_0$.    This completes the proof of (ii) of Theorem \ref{theorem1221}.
 \qed

%%%%%%%%%%%%%%%%%%%%%%%%%%%%%%%%%%%%%%%%%%%%%%%%%%%%%%%%%%%
%%%%%%%%%%%%%%%%%%%%%%%%%%%%%%%%%%%%%%%%%%%%%%%%%%%%%%%%%%%

 %%%%%%%%%%%%%%%%%%%%%%%%%%%%%%%%%%%%%%%%%%%%%%%%%%%%%
%%%%%%%%%%%%%%%%%%%%%%%%%%%%%%%%%%%%%%%%%%%%%%%%%%%%%
%%%%%%%%%%%%%%%%%%%%%%%%%%%%%%%%%%%%%%%%%%%%%%%%%%%%%%

\subsection{Proof of Theorem \ref{theorem260403} }
\label{mainsection2}

To prove Theorem \ref{theorem260403}, we use the following lemma.
\begin{lemma}\label{theorem1221-2}
Let $\phi \in C^2([0,2])$ satisfy Assumption \ref{assume1} and $ \phi'(0) =0$. Let $M$ satisfy \eqref{260321-7}.  
\begin{itemize}
\item[(i)]
Let $ 1 \leq i \leq n-1$. There exist positive constants $ c_{*5}>2$ and $c_i>0$  such that  if $c_{*5} < |x'| < c_i x_i $,  then $ D_{x_n} w_i(x', 0,t)$ is strictly decreasing  in $ t \in (1,2)$.

\item[(ii)] There exists  a  positive constant $c_{*6}>0$ such that   if $c_{*6} < |x'|   $,  then $ D^2_{x_n} w_n(x', 0,t)$ is strictly decreasing in $t \in (1,2)$.
\end{itemize}

\end{lemma}

 \begin{proof}
Since the proofs of (i) and (ii) are similar, we only prove (ii). Using the integration by parts, we have 
\begin{align}\label{0904-1}
\begin{split}
\int_0^t   \frac{\phi(s)}{(t-s)^\frac32}     D_{x_k}  f_k (x', t-s) 
%& =    2\int_0^t  \phi(s)    \frac{d}{ds}   (t-s)^{-\frac12}   f(x', t-s)  ds\\
=  -  2\int_0^t    \Big(  \phi' (s)      D_{x_k}  f _k (x', t-s)  -  \phi (s)      \frac{d}{dt} D_{x_k}  f _k (x', t-s)  \Big)  (t-s)^{-\frac12} ,
\end{split}
\end{align}
where $f_k$ is defined in \eqref{defoff}. 
Since $ \big( D_t -\De' \big) \Ga'(y', t) =0$ for $ t> 0$ and   $\int_{\Rn} \Ga' (x' - y', t) dy' =1$ for all  $(x', t)$, we have
\begin{align}\label{0904-2}
\begin{split}
\frac{d}{dt}  D_{x_k}  f _k ( x', t)  &  =\int_{\Rn} \De'_{y'} \Ga' (x'-y', t) \Big( D_{y_k} R'_k \psi(y')  - D_{x_k} R'_k \psi(x') \Big) dy'\\  
&  =\int_{\Rn}  \Ga' (x'-y', t)  \De'_{y '} D_{y_k} R'_k \psi(y')  dy'\\ 
&  =\int_{\Rn}   \Ga' (x'-y', t) \Big(  \De'_{y'} D_{y_k}   R'_k\psi(y')  -  \De'_{x'} D_{x_k} R'_k \psi(x') \Big) dy' + \De'_{x'} D_{x_k} R'_k \psi(x')\\
& := g_k(x', t) +   \De'_{x'} D_{x_k} R'_k \psi(x'),
\end{split}
\end{align}
where 
\begin{align*}
g_k (x', t-s)  = \int_{\Rn} \Ga' (x'- y', t-s) \Big(  \De'  D_{y_k} R'_k \psi(y') -   \De' D_{x_k}  R'_k \psi(x') \Big) dy'.
\end{align*} 
From \eqref{0904-1} and \eqref{0904-2},  we have 
\begin{align*}%\label{260106-1}
\begin{split}
&\int_0^t   \frac{\phi(s)}{(t-s)^\frac32}  D_{x_k}  f _k (x', t-s) ds\\
 =& -  2\int_0^t    \Big(  \phi' (s)    D_{x_k}  f _k (x', t-s)  -  \phi (s)   \big(g_k (x', t-s) +   \De'_{x'} D_{x_k} R'_k \psi(x')    \big)  \Big)  (t-s)^{-\frac12}   ds\\
=& - 4\int_0^t    \Big(  \phi'' (s)    D_{x_k}  f _k(x', t-s)  -  2\phi'  (s)   \Big(g_k (x', t-s) +   \De'_{x'} D_{x_k} R'_k \psi(x') \Big) \\
& \quad    + \phi (s) \big(h_k (x', t-s)   + ( \De')^2 D_{x_k}  R'_k \psi(x')   \big) \Big)  (t-s)^{\frac12}   ds,
\end{split}
\end{align*}

where 
\begin{align*}
h_k (x', t-s)  = \int_{\Rn} \Ga' (x'- y', t-s) \Big(  ( \De')^2  D_{y_k} R'_k \psi(y') -  ( \De')^2  D_{x_k}  R'_k \psi(x') \Big) dy'.
\end{align*}

Reminding that $ {\rm div} \, w =0$, we have  from \eqref{shortg}
\begin{align}\label{260318-1-1}
\begin{split}
D_{x_n}^2 w_n(x',  0, t) & =  -\sum_{1 \leq k \leq n-1} D_{x_k} D_{x_n} w_k (x', 0, t)\\
 & =   c_n \sum_{1 \leq k \leq n-1}\int_0^t  \frac{ \phi(s)}{(t-s)^\frac32}    D_{x_k}  f_k (x', t-s) ds    -c_n  \sum_{1 \leq k \leq n-1} D_{x_k}  R'_k \psi(x) M(t)\\
 & = -4  c_n  \sum_{1 \leq k \leq n-1} \int_0^t    \Big(  \phi'' (s)    D_{x_k}  f _k(x', t-s)  -  2\phi'  (s)   [g_k (x', t-s) +   \De'_{x'} D_{x_k} R'_k \psi(x') ] \\
& \quad    + \phi (s) \big(h_k (x', t-s)   + ( \De')^2 D_{x_k}  R'_k \psi(x')   \big) \Big)  (t-s)^{\frac12}   ds\\
 &  \quad    -4c_n  \sum_{1 \leq k \leq n-1} D_{x_k}  R'_k \psi(x) M(t).
 \end{split}
\end{align}
Hence,  it follows from \eqref{260318-1-1} that
\begin{align}\label{260318-1}
\begin{split}
\frac{d}{dt} D_{x_n}^2 w_n(x',  0, t) 
 & =  -4  c_n  \sum_{1 \leq k \leq n-1} \int_0^t    \Big(  \phi'' (s)     \frac{d}{dt}  D_{x_k}  f _k(x', t-s)  -  2\phi'  (s)   \frac{d}{dt}  g_k (x', t-s)  \\
& \quad    + \phi (s)  \frac{d}{dt}  h_k (x', t-s)   \Big)  (t-s)^{\frac12}   ds\\
&\quad   -2    c_n  \sum_{1 \leq k \leq n-1} \int_0^t    \Big(  \phi'' (s)    D_{x_k}  f _k(x', t-s)  -  2\phi'  (s)   \Big(g_k (x', t-s) +   \De'_{x'} D_{x_k} R'_k \psi(x') \Big) \\
& \quad    + \phi (s) \big(h_k (x', t-s)   + ( \De')^2 D_{x_k}  R'_k \psi(x')   \big) \Big)  (t-s)^{-\frac12}   ds\\
 &  \quad    -2c_n  \sum_{1 \leq k \leq n-1} D_{x_k}  R'_k \psi(x)M'(t).
 \end{split}
\end{align}

Note that  for  $  2 < | x'| $.
\begin{align}\label{260318-3}
\sum_{1 \leq k \leq n-1} D_{x_k}   R'_k \psi(x')  & = -    \frac1{n\om_n } \int_{\Rn}  \frac{1} {|x' -z'|^n}   \psi(z')   dz'  \approx - |x'|^{-n} c_\psi.
\end{align}
From the proof of Lemma \ref{lemma1116-1},  for $ 2 < |x'| $,
\begin{align}\label{260404-6}
\begin{split}
|     D_{x_k} f_k ( x', t)|, \quad  |x'|^2|     g_k ( x', t)|, \quad |x'|^4 | h_k (x', t) | & \leq   c |x'|^{-n-2} t,\\
|\frac{d}{d t}  D_{x_k}  f _k ( x', t)|,\quad |x'|^2 |\frac{d}{d t}  g _k ( x', t)|,\quad  |x'|^4 |\frac{d}{d t} h_k ( x', t)|  &  \leq c |x'|^{-n-2}.
\end{split}
\end{align}

Applying    \eqref{260318-3} and \eqref{260404-6} to  \eqref{260318-1}, if $ M'(t) < -e_0$, then  for $ 1 < t < 2$,  we obtain 
 \begin{align}\label{0915-1-1}
 \begin{split}
  \frac{d}{dt} D^2_{x_n} w_n (x', 0, t) & \leq  c|x'|^{-n-2} \int_0^t \Big( |\phi(s)| + |\phi'(s)|  + |\phi''(s)|   \Big) ( t-s)^{\frac12} ds \\
& \quad +  c|x'|^{-n-2} \int_0^t \Big( |\phi(s)| + |\phi'(s)|     \Big) ( t-s)^{-\frac12} ds+ |x'|^{-n } M'(t)\\
  & \leq   c_2 |x'|^{-n-2}   - c e_0 |x'|^{-n}\\
  & < 0
 \end{split}
 \end{align}
provided $ \frac{c_2}{c e_0}   < |x'|^2 $.
This implies  that $ D^2_{x_n}  w_n (x', 0, t) $ is strictly decreasing for $ t \in (1,2)$. This completes the proof of (ii) of 
Theorem \ref{theorem1221-2}.
\end{proof}

 \begin{pfthm-SS-19}
Since the proofs of (i) and (ii) are similar, we only prove (ii).   Applying  \eqref{260322-1},   \eqref{260318-3} and $\eqref{260404-6}_1$ to \eqref{260318-1-1},   if $ 0 < t< 1$, then for large $|x'|$, we have 
\begin{align*}%\label{260319-1-1}
D^2_{x_n} w_n (x',0,t) \geq    - c_1 c_\psi \phi(t)      |x'|^{-n-2} + c_2  c_\psi\phi(t)    |x'|^{-n } > 0.
\end{align*}

From   Assumption \ref{assume1},   we have 
  \begin{align*}%\label{260319-2-1}
\begin{split}
D_{x_n}^2 w_n(x',  0, t)  
 & =   c_n \sum_{1 \leq k \leq n-1}\int_0^t  \frac{ \phi(s)}{(t-s)^\frac32}    D_{x_k}  f_k (x', t-s) ds    -c_n  \sum_{1 \leq k \leq n-1} D_{x_k}  R'_k \psi(x) M(t)\\
&  \leq    - c_3 c_\psi \Psi(t)    |x'|^{-n-2} +   c_4 c_0  c_\psi |x'|^{-n } M(t), 
\end{split}
\end{align*}
where $\Psi(t) = \int_0^t \frac{\phi(s)}{(t -s)^\frac12} ds$. In particular,  since $\phi(2) =0$, we have 
\begin{align*}
\begin{split}
D_{x_n}^2 w_n(x',  0, 2)  &  <     - c_3 c_\psi \Psi(t)    |x'|^{-n-2}<0,
\end{split}
\end{align*}
Let $ x'$ fix satisfying $|x'| > c_*$.  Since $D_{x_n}^2 w_n(x',0,t)$ is strictly decreasing on $(1,2)$ by Lemma \ref{theorem1221-2}(ii), there exists a unique $t_n^*(x')\in(1,2)$ such that
\[
D_{x_n}^2 w_n(x',0,t)>0 \quad \text{for } t\in(0,t_n^*(x')),
\qquad
D_{x_n}^2 w_n(x',0,t)<0 \quad \text{for } t\in(t_n^*(x'),2).
\]
Then $(x',t_n^*(x'))$ is a point of boundary layer separation for $w_n$.
Since $D_{x_n}p(x',0,t)=D_{x_n}^2 w_n(x',0,t)$, we obtain (ii), and thus we complete the proof of (ii).

Finally, let $x' = x_1 \mathbf{e}_i$ with $x_1 > c_*$. By the same argument as in the proof of Theorem \ref{theorem1221}, we show that $t_n^*(x')$ is monotone in the $\mathbf{e}_i$-direction. Moreover, since $\psi$ is radially symmetric, both $D_{x_n}^2 w_n(x',0,t)$ and $t_n^*(x')$ are radially symmetric in $x'$, and hence depend only on $|x'|$. This completes the proof for all directions, and thus yields (iii).

Assertion (iv) follows by the same argument as in the proof of Theorem \ref{theorem1221}(ii). This completes the proof of Theorem \ref{theorem260403}.
\end{pfthm-SS-19}

%%%%%%%%%%%%%%%%%%%%%%%%%%%%%%%%%%%%%%%%
%%%%%%%%%%%%%%%%%%%%%%%%%%%%%%%%%%%%%%%%
%%%%%%%%%%%%%%%%%%%%%%%%%%%%%%%%%%%%%%%%

\section{Proof of Theorem \ref{maintheorem(NS)}}\label{NS-thm}
%\section{ Holder continuity of  solutions of Stokes equations}
\setcounter{equation}{0}
In this section, we provide the proof of Theorem \ref{maintheorem(NS)}

We first  investigate the Hölder continuity of solutions to the following system:
\begin{align}\label{Stokes-10-2}
\left\{\begin{array}{l}\vspace{2mm}
 v_t - \De v+  \na q ={\rm div } F, \quad Q_+(T),\\
\vspace{2mm}
\mbox{div } v =0,  \quad  Q_+(T),\\
v|_{x_n} =g, \,\, v|_{t =0} =0. 
\end{array}
\right.
\end{align}

\begin{lemma}\label{so}
Let  $w$ be a solution to  the system \eqref{Stokes-10-2} with   $ F \equiv 0$. For $ k \in \{1,2\}$ and $ 0 < \al < 1$, assume that
   $g, \, R' g_n  \in   C^{ k +\al, \frac{k}2 +\frac{\al}2} (q(T)) $ and the compatibility conditions $  \sum_{ 0 \leq l \leq k-1}  D^{l}_t g(x',0)   =0$ hold.  Then,    $w$ satisfies:
\begin{align*}
\| w\|_{    C^{ k +\al, \frac{k}2 +\frac{\al}2}  (Q_+(T) )} &  \leq c \Big( \| g\|_{   C^{ k +\al, \frac{k}2 +\frac{\al}2} (q(T)) }  
+ \|R' g_n \|_{  C^{ k +\al, \frac{k}2 +\frac{\al}2} (q(T) ) }  \Big).
\end{align*}
\end{lemma}
The proof of Lemma \ref{so} is provided in Subsection \ref{existenceofsolution2} in Appendix.

\begin{lemma}\label{theo1123}
Let  $w$ be a solution to  the system \eqref{Stokes-10-2} with  $ F \equiv 0$.   Assume that $D_{x'}R' g'$ and $D_{x'}  g'$ belong to  $L^\infty(\Rn;  C^{ \frac{\al}2}( 0, T))$. Then $w$ satisfies
\begin{align*} %\label{1212-1}
\begin{split}
|D_{x_n}w_i (x',0,t)|  & \leq c  t^\frac{\al}2 \Big(  \|  g\|_{ C^{1 +\al,\frac12 +\frac{\al}2}(q (  T))}  +  \| R' g_n \|_{ C^{1 +\al,\frac12 +\frac{\al}2}(q ( T))}     \Big).
\end{split}
\end{align*}
\end{lemma}
The proof of Lemma \ref{theo1123} is deferred to Subsection \ref{prooftheo1123}  in Appendix.

\begin{lemma}\label{prop1120}
For fixed  $ k \neq n$, let us define
\begin{equation*}%\label{1120-1}
V^1(x,t) =\int_0^t \int_{\R} D_{x_k} \Ga(x-y, t -s) f(y,s) \,  dyds.
\end{equation*}
For $ f \in  L^\infty(0, T; L^\infty(0, \infty; \dot C^\al(\Rn)))$ with $  0<\al <1$,  the following estimates hold:
\begin{equation*}%\label{1118-3-4}
\begin{aligned}
\|  R' V^1 \big|_{x_n =0}\|_{L^\infty(0, T; \dot C^\al(\Rn))} &  \leq cT^{\frac12 +\frac{\al}2} \|  f \|_{L^\infty(0, T; L^\infty(0, \infty; \dot C^\al(\Rn)))},\\
\| R'  V^1 \big|_{x_n =0}  \|_{L^\infty( \Rn; \dot C^{\frac12 +\frac{\al}2}(0, T))} &  \leq c   \| f \|_{L^\infty(0, T; L^\infty(0, \infty; \dot C^\al(\Rn)))},\\
\| R' V^1 \big|_{x_n =0} \|_{L^\infty (\Rn \times (0, T))} &  \leq cT^{\frac12 +\frac{\al}2} \| f \|_{L^\infty(0, T; L^\infty(0, \infty; \dot C^\al(\Rn)))}.
\end{aligned}
\end{equation*}
\end{lemma}
 The proof of Lemma \ref{prop1120} is referred to  Subsection \ref{proofprop1120} in Appendix.

\begin{lemma}\label{theo1118-1}
Let $F \in C^{ 2 +\al,1 +\frac{\al}2 }(Q_+( T))$ and $ g, \, R' g_n \in C^{2+\al, 1 +\frac{\al}2}(q ( T))$ satisfy the compatibility conditions $ g(x',0) = D_t g(x',0) =0$. Then, there exists a solution $v $ to   \eqref{Stokes-10-2} such that 
 \begin{equation*}
\begin{aligned}
\| v\|_{    C^{ 2 +\al, 1 +\frac{\al}2}  (Q_+( T) )} &  \leq c \Big( \| g\|_{   C^{ 2 +\al, 1 +\frac{\al}2} (q ( T)) }  
+ \|R' g_n \|_{  C^{ 2 +\al, 1 +\frac{\al}2} (  q (  T)) }  +  \max(1, T^{\frac12 +\frac{\al}2})  \| F\|_{  C^{2 + \al, 1+ \frac{\al}2} (Q_+(T))} \Big),\\
| D_{x_n}  v(x', 0,t)| &  \leq c t^{\frac{\al}2} \Big( \| g\|_{L^\infty(0, t; C^{1 + \al} (\Rn))} +  \|R' g_n \|_{L^\infty(0, t; C^{1 + \al} (\Rn))} 
+ \| F\|_{L^\infty(0, t; C^{1 + \al} (\R_+))}   \Big).
\end{aligned}
\end{equation*}
\end{lemma}

\begin{proof}
Let $\tilde F  \in C^{\al, \frac{\al}2} (Q ( T))$ be an extension of $F$  to the whole space $Q(T)$ such that  $ \tilde F|_{\R_+ \times (0, T)} =F$ and  $\| \tilde F\|_{C^{\al, \frac{\al}2} (Q( T))} \leq c \|  F\|_{C^{\al, \frac{\al}2} (Q_+( T))}$.
Consider the following system defined in whole-space:
\begin{align*}%\label{Stokes-10-3}
\left\{\begin{array}{l}\vspace{2mm}
 V_t - \De V+  \na Q ={\rm div } \tilde F, \quad Q( \infty),\\
\vspace{2mm}
\mbox{div } V =0 \quad Q( \infty),\\
  V|_{t =0} =0. 
\end{array}
\right.
\end{align*}
The solution $V$ is given by the formula:
\begin{align*}
V_i (x,t) &  = \sum_{1 \leq k \leq n} \int_0^t \int_{\R} D_{x_k} \Ga (x-y, t-s) \big( \de_{ik} + R_i R_k   \big) \tilde F_{ik}(y, s) dyds, \quad 1 \leq i \leq n.
\end{align*}
It is well-known that 
\begin{align*}%\label{1118-5}
\| V \|_{\dot C^{2+\al, 1+\frac{\al}2} (Q ( \infty))} &\leq c \| F\|_{L^\infty(0, \infty; \dot C^{ 1 + \al} (\R \times (0, \infty))}.
\end{align*}

Recall that  $D_x \Ga(x,t)$ belongs to the  Hardy space $H^1(\R)$ with a norm bound of $c t^{-1/2}$.   Since the Riesz transform $R$ is bounded on $BMO(\mathbb{R}^n)$, and $BMO$ is the dual of $H^1$, we have
\begin{equation}\label{251208-1}
\begin{aligned}
| D_x V(t, x)| &  \leq c \int_0^t \|D_x \Ga (t -s )\|_{H^1(\R)} \Big(  \| D_x R \tilde F (s) \|_{BMO (\R)} +  \| D_x \tilde F (s) \|_{BMO (\R)} \Big)  ds\\
& \leq c \int_0^t (t -s)^{-\frac12  }  \| D_x \tilde F (s) \|_{L^\infty (\R)} ds\\
& \leq c t^{\frac12 } \| D_x \tilde F  \|_{L^\infty (Q ( t)) }.
\end{aligned}
\end{equation}
Hence,  $
\|D_x V\|_{L^\infty (Q (  T)) } 
 \leq c  T^{\frac12 } \|D_x \tilde  F\|_{L^\infty (Q ( T))   }$. Similarly, $
\| D_x^2 V\|_{L^\infty (Q ( T))  } \leq c T^\frac12 \| D_x^2 \tilde F\|_{L^\infty (Q (  T)) }$.

Following the proof of Proposition   4.4 in  \cite{CJ2},   $V_n$  and its derivatives satisfy:
\begin{equation}\label{1118-4}
\begin{aligned}
 V_n (x,t) & = \sum_{l \neq n} \int_0^t \int_{\R} D_{x_l} \Ga (x-y, t-s)  \big( \de_{nl} + R_n R_l  \big) \tilde F_{nl}(y, s) dyds,\\
 D^2_x V_n (x,t) & = \sum_{l \neq n} \int_0^t \int_{\R} D_{x_l} \Ga (x-y, t-s)  D^2_y \big( \de_{nl} + R_n R_l  \big) \tilde  F_{nl}(y, s) dyds,\\
D_t V_n (x,t)  & = \sum_{ l \neq n}   D_{x_l}  \big( \de_{nl} + R_n R_l   \big) \tilde  F_{nl}(x, t)\\
& \quad +  \sum_{ k \neq n } \int_0^t \int_{\R} D_{x_k} \Ga (x-y, t-s)  \De_y  \big( \de_{nk} + R_n R_k   \big) \tilde  F_{nk}(y, s) dyds.
\end{aligned}
\end{equation}
Note that 
\begin{align*}
\| f \|_{L^\infty(0, T; L^\infty(0, \infty; \dot C^\al(\Rn)))} \leq c \| f \|_{L^\infty(0, T; L^\infty( \dot C^\al(\R_+)))} \leq c \| f \|_{  \dot C^{\al,\frac{\al}2}(Q_+(T))}.
\end{align*}
Hence, from \eqref{1118-4} and  Lemma \ref{prop1120}, for  $    0<\al <1$,  we have 
\begin{align*}
\| R' V_n|_{x_n =0}\|_{C^{2 + \al,  1 + \frac{\al}2} (q ( T))} &  \leq c  \max( 1, T^{\frac12 +\frac{\al}2} )  \|  \tilde F \|_{C^{2 + \al,  1 + \frac{\al}2} (Q_+ ( T))}.
\end{align*}

Let $ (w, p) $ be a solution of equations \eqref{Stokes-10} with boundary data $ g - V|_{x_n =0}$.  Using Lemma \ref{so},  $w$ satisfies  the following estimates:
\begin{align*}
\begin{split}
\| w\|_{    C^{ 2 +\al,1+\frac{\al}2}  (Q_+(T) )} &  \leq c \Big( \| g - V|_{x_n =0}\|_{   C^{2+\al, 1+\frac{\al}2} (q ( T)) }  
+ \|R' \big( g_n - V|_{x_n =0}  \big) \|_{  C^{2 +\al, 1 +\frac{\al}2} (q (  T)) }  \Big)\\
& \leq c \Big(   \| g  \|_{   C^{2 +\al, 1 +\frac{\al}2} (q (  T)) }  
+ \|R'   g_n  \|_{  C^{2+ \al,1 +\frac{\al}2} (q (  T)) }   +  \max(1, T^{\frac12 +\frac{\al}2})  \| F\|_{  C^{2 + \al, 1+ \frac{\al}2} (Q_+(T))} \Big).
\end{split}
\end{align*}

From Theorem \ref{theo1123} and \eqref{251208-1}, we have 
\begin{align*}
| D_{x'}  w(x', 0,t)|  &  \leq c t^{\frac{\al}2} \Big( \| g\|_{L^\infty(0, t; C^{1 + \al} (\Rn))} +  \|R' g_n \|_{L^\infty(0, t; C^{1 + \al} (\Rn))} 
+   \| F\|_{L^\infty(0, t; C^{1 + \al} (\R_+))}   \Big).
\end{align*}
Finally, setting  $ v = V + w$  and applying Lemma \ref{theo1123} and \eqref{251208-1}, we conclude the desired estimates. This completes the proof of Lemma \ref{theo1118-1}.
\end{proof}

 Next we study the perturbed Navier-Stokes equations near the solution of the Stokes system.
 
%\section{Holder continuity of solutions of Navier-Stokes equations}
%\setcounter{equation}{0}

Let $g, \, R' g_n \in C^{2 +\al, 1 +\frac{\al}2}(q(T))$ and let $w$ be a solution to the Stokes system  with boundary data $ag$, where $a >0$ is determined later. 
We define: 
\begin{align*}
M_0 =\| g\|_{ C^{2 +\al, 1 +\frac{\al}2}(q(T)) } + \| R'g_n \|_{ C^{2+\al, 1 +\frac{\al}2}(q(T)) }.
\end{align*}  

We consider the following nonlinear system for $(v, \pi)$:
\begin{align} \label{CCK-Feb7-20}
\begin{split}
&v_t-\Delta v+\nabla \pi =-{\rm div}\,\left(v\otimes v+v\otimes w+w\otimes v+w\otimes w \right),\quad {\rm div} \, v =0, \quad Q_+(T),\\
& v (x,0)=0, \quad v (x', 0,t)=0.
\end{split}
\end{align}
The solution $u$ to the Navier-Stokes equations with boundary data $ag$ is then decomposed as $u = w + v$.

In next proposition, we construct regular the solution $v$  in \eqref{CCK-Feb7-20}, provided that $a$ is sufficiently small. 
\begin{proposition}\label{theorem1118-3}
Let $0 < \ep_0 < 1$. 
If $a M_0< c \ep_0$, then the system \eqref{CCK-Feb7-20} has a unique solution $(v, \pi)$ satisfying the following estimates:
\begin{align}\label{251209-7}
\| v\|_{ C^{2 +\al, 1 +\frac{\al}2} ( Q_+(T)) } & < \ep_0 a M_0,
\end{align}
\begin{align}\label{251208-6}
| D_{x'} \pi (x', 0, t)| & \leq c \ep_0 a M_0 t^{\frac{\al}2}.
\end{align}
\end{proposition}

\begin{proof}
First, standard Stokes estimates imply $\|w\|_{C^{2+\alpha, 1+\frac{\alpha}{2}}} \leq c a M_0$. By absorbing the constant into $M_0$, we assume $\|w\| \leq a M_0$. We define an iterative scheme: for $m \geq 0$, let $v^{m+1}$ solve the linear Stokes system with the nonlinear terms of the previous step as the external force:
\begin{align*}
&v^{m+1}_t-\Delta v^{m+1}+\nabla \pi^{m+1}=-{\rm div}\,\left(v^{m}\otimes v^{m}+v^{m}\otimes w+w\otimes v^{m}+w\otimes w \right), \qquad {\rm div} \, v^{m+1} =0
\end{align*}
with homogeneous initial and boundary data, i.e., $v^{m+1}(x,0)=0$ and $v^{m+1}(x', 0,t)=0$.

Setting $v^0 = 0$, it follows from Lemma  \ref{theo1118-1} that: 
\begin{align*}%\label{251213-1}
\| v^1\|_{ C^{2+\al, 1 +\frac{\al}2} ( Q_+ (T) ) } \leq c \max(1, T^{\frac12 +\frac{\al}2}) \| w\|^2_{ C^{2 + \al, 1+ \frac{\al}2} ( Q_+ (T) ) } \leq c(T) a^2 M_0^2.
\end{align*}

Choosing $a$ such that $2c(T) a M_0 < \epsilon_0 < 1/3$, we obtain $\| v^1 \|_{ C^{2+\al, 1 +\frac{\al}2} ( Q_+ (T) ) } \leq \frac{1}{2} \epsilon_0 a M_0 < \epsilon_0 a M_0$. By induction, suppose $\| v^m \| \leq \epsilon_0 a M_0$. Then, from   Lemma \ref{theo1118-1}, we have  
\begin{align*}
\begin{split}
\| v^{m+1}\|_{ C^{2+\al, 1 +\frac{\al}2} ( Q_+ (T) ) } 
& \leq c(T) \Big( a^2 M_0^2 + 2a M_0 \| v^m \|_{ C^{2 +\al, 1 +\frac{\al}2} ( Q_+ (T) ) } + \| v^m\|^2_{ C^{2 +\al, 1 +\frac{\al}2} ( Q_+ (T) ) } \Big)\\
& \leq c(T) \Big( a^2 M_0^2 + 2\ep_0 a^2 M_0^2 + \ep_0^2 a^2 M_0^2 \Big)\\
& = c(T) a^2 M_0^2 (1 + \ep_0)^2.
\end{split}
\end{align*}
Since $(1+\epsilon_0)^2 < 2$ for $\epsilon_0 < 1/3$, our choice of $a$ ensures $\| v^{m+1} \|_{ C^{2+\al, 1 +\frac{\al}2} ( Q_+ (T) ) } \leq \epsilon_0 a M_0$.

Next, we denote $V^{m+1}:=v^{m+1}-v^{m}$ and $Q^{m+1}:=q^{m+1}-q^{m}$ for $m\ge 1$. We then see that $(V^{m+1}, Q^{m+1})$ solves
\begin{align*}
\begin{split}
& V^{m+1}_t-\Delta V^{m+1}+\nabla Q^{m+1}=-{\rm div}\,\left(V^{m}\otimes v^{m}+v^{m-1}\otimes V^{m}+V^{m}\otimes w+w\otimes V^{m}\right),\\
& {\rm div} \, V^{m+1} =0,
\end{split}
\end{align*}
with homogeneous initial and boundary data, i.e., $V^{m+1}(x,0)=0$ and $V^{m+1}(x',0,t)=0$. 

Then, we obtain
\begin{align*}%\label{0531-1-3}
\begin{split}
\| V^{m+1}\|_{ C^{2 +\al, 1 +\frac{\al}2} ( Q_+ (T) ) } 
  \leq & c(T) \Big( \| V^m\|_{ C^{2 +\al, 1+\frac{\al}2} ( Q_+ (T) ) } \| v^m \|_{ C^{2 +\al, 1+\frac{\al}2} ( Q_+ (T) ) } \\
  &\quad + \| V^{m}\|_{ C^{2 +\al, 1 +\frac{\al}2} ( Q_+ (T) ) } \| v^{m-1} \|_{ C^{2+\al, 1+\frac{\al}2} ( Q_+ (T) ) }\\
  & \quad + \| V^{m}\|_{ C^{2 +\al, 1 +\frac{\al}2} ( Q_+ (T) ) } \|w \|_{ C^{2 +\al, 1 +\frac{\al}2} ( Q_+ (T) ) }\Big)\\
 \leq & c(T) \Big(2\ep_0 a M_0 + a M_0 \Big) \| V^m\|_{ C^{2 +\al, 1 +\frac{\al}2} ( Q_+ (T) ) }\\
  \leq &\frac12 \| V^m\|_{ C^{2 +\al, 1 +\frac{\al}2} ( Q_+ (T) ) }.
\end{split}
\end{align*}
 
Therefore, there exists a limit $v \in C^{2 +\al, 1 +\frac{\al}2} ( Q_+ (T) ) $ such that $v^m \rightarrow v$ in $C^{2 +\al, 1 +\frac{\al}2} ( Q_+ (T) )$.   

Then, $v$ becomes a strong solution of the Navier-Stokes equations \eqref{CCK-Feb7-20} with boundary data zero, satisfying
\begin{align*}%\label{1123-5}
\| v\|_{ C^{2+\al, 1 +\frac{\al}2}( Q_+ (T) ) } & \leq \ep_0 a M_0. 
\end{align*}

Finally, for the pressure estimate, taking the trace on the boundary $\{x_n = 0\}$ where $v=0$, and utilizing the divergence-free condition ${\rm div } \, w = 0$, we have
\begin{align*} 
| D_{x'} \pi (x', 0, t)| \leq c | D_{x_n}^2 v(x', 0, t) | \leq c \ep_0 a M_0 t^{\frac{\al}2}.
\end{align*}
This completes the proof of Proposition \ref{theorem1118-3}.
\end{proof}

%\section{Proof of Theorem \ref{maintheorem(NS)}}
%\setcounter{equation}{0}

Now we are ready to present the proof of Theorem \ref{maintheorem(NS)}.

\begin{pfthm-NS}
\\
Since $u = w + v$, it follows from Lemma  \ref{so} and \eqref{251209-7} in Lemma  \ref{theorem1118-3} that $u \in C^{2+\al, 1 +\frac{\al}2}(Q_+ (T))$ and    so $D_x v \in C^{1+\al, \frac12+\frac{\al}2}(Q_+ (T))$. Because $D_{x_n} v(x', 0, 0) = 0$, by Proposition \ref{theorem1118-3},  we have
\begin{align}\label{260410-3}
|D_{x_n} v_i(x', 0, t)| \leq t^{\frac12 +\frac{\al}2} \|v\|_{C^{2+\al, 1 +\frac{\al}2} (Q_+ (T))} \leq \ep_0 a M_0 t^{\frac12 +\frac{\al}2}.
\end{align}
From \eqref{shortg},   \eqref{1026-6-1}, $\eqref{asumption260410}_1$  and Lemma \ref{lemma1116-1}, for $ 0 < t< 1$ and for $ c_* < |x'|$,  we have 
\begin{align} \label{260410-4} 
\begin{split}
D_{x_n} w_i (x', 0,t)  & =  
   - c_n a  \int_0^t    \frac{ \phi (s)}{(t -s)^\frac32}       f(x', t-s)  +  4 c_n a  R'_i \psi (x') \int_{0}^t   \phi' (s)     (t-s)^{-\frac12}  ds\\
& \geq  - c_1 a  |x'|^{-n -1}  t^\frac32 + c_2 a  |x'|^{-n+1} t^{\frac12 +\frac{\al}2}\\
& \geq  c  a |x'|^{-n+1} t^{\frac12 +\frac{\al}2}.
\end{split}
\end{align}
 Combining \eqref{260410-3} and \eqref{260410-4}, we have 
\begin{align}
\label{1118-7} D_{x_n} u_i (x', 0, t) &\geq  a t^{\frac12 + \frac{\al}2} \Big( c_1  |x'|^{-n+1} -   d \ep_0 M_0\Big), \quad 0 < t < 1.
\end{align}
%{\color{red}{ $\ep_0 < \frac{c_1   c_{**}^{-n +1}}{dM_0}$   }}
On the other hand, using   \eqref{260319-2} and \eqref{251209-7}, we obtain
\begin{align}
\label{1118-8} D_{x_n} u_i (x', 0, t) &\leq - c_2 a   |x'|^{-n +1}   + c_4 |x'|^{-n +1} M(t) +   d \ep_0 a M_0, \quad 1 < t < 2,
\end{align}

%{\color{red}{  $\ep_0 < \frac{c_2  c_{**}^{-n +1}}{dM_0}$      }}
Furthermore, applying \eqref{0915-1-1}  to $ D_{x_n} w_i$, then   from \eqref{251209-7}, for $1 < t < 2$, we have
\begin{align}\label{260417-1}
\frac{d}{dt}  D_{x_n} u_i(x', 0,t) < - c_3   a  |x'|^{-n +1} +  d  \ep_0 a M_0.
\end{align}
%{\color{red}{   $\ep_0 < \frac{c_3   c_{**}^{-n +1}}{dM_0}$    }}
It also follows from %\eqref{eq_proof1}, 
\eqref{pressure260410}  and \eqref{251208-6} that
\begin{align}\label{260417-2}
\begin{split}
%D_{x_i} q (x', 0,t) &= D_{x_i} p (x', 0,t) + D_{x_i} \pi (x', 0,t) \\
%&\leq  - c_4 a   |x'|^{-n +1} t^{\frac{\al}2}   + d \ep_0 a M_0 t^{\frac{\al}2}, \quad  0 < t<t_1,\\
D_{x_i} q (x', 0,t) & \geq  c_4  a\phi(t)|x'|^{-n -1}  -  c_5 a \phi' (t) |x'|^{-n +1}  -  d  \ep_0 a M_0, \quad 1 < t<2
\end{split}
\end{align}
for large $|x'|$.
%{\color{red}{  %$\ep_0 < \frac{c_4   c_{**}^{-n +1}}{dM_0}$ and  
%$\ep_0 < \frac{c_5  \phi(t)  c_{**}^{-n -1} + c_6  | \phi'(t)|  c_{**}^{-n +1}}{dM_0}  $     }}

Let 
\[
\varepsilon_0 \le \frac12 \min\left\{
\frac{c_{**}^{-n+1}}{dM_0}\min_{1\le i\le 5} c_i,\;
\frac{c_{**}^{-n+1}}{dM_0}\min_{1<t<2}\left(c_5 \phi(t)c_{**}^{-2}+c_6|\phi'(t)|\right)
\right\}.
\]
From  $\eqref{asumption260410}_2$, we have $\ep_0 > 0$.  Then, for $ c_* < |x'| < c^{**}$,  we obtain
\begin{align*}
D_{x_n} u_i (x', 0, t) &> 0, \quad 0 < t < 1, \quad \mbox{from} \quad \eqref{1118-7}, \\
D_{x_n} u_i (x', 0, 2) &< 0, \quad  \mbox{from} \quad \eqref{1118-8}, \\
\frac{d}{dt}  D_{x_n} u_i(x', 0,t) &< 0, \quad 1 < t < 2, \quad \mbox{from} \quad \eqref{260417-1},\\
%D_{x_i} q (x', 0,t) &< 0, \quad  0 < t < t_1,\\
D_{x_i} q (x', 0,t) &> 0, \quad 1 < t<2, \quad \mbox{from} \quad \eqref{260417-2}.
\end{align*}
By the same argument as in Section \ref{mainsection}, there exists $1 < t^* < t_0$ such that $(x', t^*)$ is a point of boundary layer separation for $u_i$.
This completes the proof of (i)  of Theorem \ref{maintheorem(NS)}.

 Lets remember  \eqref{1118-7}, $D_{x_n} u_i (x', 0, t) >0$.  From  \eqref{060410-5}  and \eqref{251209-7},   for $ 1 < t<2$, 
\begin{align*}   
\begin{split}
 D_{x_n} u_i  (x', 0,t) >      - c_1   |x'|^{-n-1}   +  c_2   |x'|^{-n+1} -    \ep_0  a M_0.
\end{split}
\end{align*}
The same arguement with the proof of (i), we have that  there is small number $\ep_0>0$ such that if $ c_* < |x'| < c_{**}$ with $ |x'| < c_i x_i$, then
$ u_i$ has no boundary layer separation with form $(x', 0, t)$ for all $ t \in  (0,2)$. We complete the proof of (ii) of Theorem \ref{maintheorem(NS)}.
\end{pfthm-NS}  

%%%%%%%%%%%%%%%%%%%%%%%%%%%%%%%%%%%%%%%%
%%%%%%%%%%%%%%%%%%%%%%%%%%%%%%%%%%%%%%%%
%%%%%%%%%%%%%%%%%%%%%%%%%%%%%%%%%%%%%%%%

\section{Appendix}\label{app-lemmas}

\setcounter{equation}{0}
\numberwithin{equation}{section}

\subsection{Proof of Lemma \ref{thoe1021-1}}\label{prooflemma1}

We note firts that $\displaystyle \int_{\Rn} D_{x_n} N(x' -y', x_n) h(y') dy' \Big|_{x_n =0} =  \frac12 h(x')$. Since 
\begin{align*}
D_{x_n} w^N_i(x,t) &= 2 \int_{\Rn} D_{x_n} N(x' -y', x_n) D_{y_i} g_n(y',t) dy', \quad 1 \leq i \leq n-1, 
\end{align*}
we have 
\begin{align}\label{w-in}
 D_{x_n} w^N_i(x', 0,t) &=  D_{x_i} g_n(x',t), \quad 1 \leq i \leq n-1.
\end{align}

Since $2\int_{-\infty}^t \int_{\Rn} D_{x_n} \Ga(x' -y', x_n, t-s) dy' ds = -1$ for all $(x,t)$ and $g_n(x',s) = 0$ for $s < 0$, we have 
\begin{align}\label{0106-1}
\begin{split}
w_i^B(x', x_n,t) +  R'_i g_n(x', t) &= 2\int_{-\infty}^t \int_{\Rn} D_{x_n} \Ga(x' -y', x_n, t-s)\big( R'_i g_n(y',s) - R'_i g_n(x',t) \big) dy' ds \\
&= 2\int_0^t \int_{\Rn} D_{x_n} \Ga(x' -y', x_n, t-s) \big( R'_i g_n(y',s) - R'_i g_n(x',s) \big) dy' ds \\
& \quad +  2 \int_{-\infty}^t \int_{\Rn} D_{x_n} \Ga(x' -y', x_n, t-s) \big( R'_i g_n(x',s) - R'_i g_n(x',t) \big) dy' ds.
\end{split}
\end{align}
Since $w_i^B(x', 0, t) = -  R'_i g_n(x', t)$, from \eqref{0106-1}, we obtain 
\begin{align*}
\begin{split}
D_{x_n} w^B_i(x', 0,t) &= \lim_{x_n \rightarrow 0} \frac{w_i^B(x', x_n, t) +  R'_i g_n(x', t)}{x_n} \\
&= - \frac{1}{ 2\pi^ \frac12} \int_0^t \frac1{(t -s)^\frac32} \int_{\Rn} \Ga'(x' -y', t-s) \big( R'_i g_n(y',s) - R'_i g_n(x',s) \big) dy' ds \\
&\quad - \frac{1}{ 2\pi^ \frac12}  \int_{-\infty}^t \frac1{(t -s)^\frac32} \int_{\Rn} \Ga'(x' -y', t-s) \big( R'_i g_n(x',s) - R'_i g_n(x',t) \big) dy' ds. 
\end{split}
\end{align*}

Similarly, we have 
\begin{align*}
\begin{split}
 D_{x_n} w^G_i(x', 0,t) &=\frac{1}{ 2\pi^ \frac12} \int_0^t \frac1{(t -s)^\frac32} \int_{\Rn} \Ga'(x' -y', t-s) \big( g_i(y',s) - g_i(x',s) \big) dy' ds \\
&\quad + \frac{1}{ 2\pi^ \frac12}  \int_{-\infty}^t \frac1{(t -s)^\frac32} \int_{\Rn} \Ga'(x' -y', t-s) \big( g_i(x',s) - g_i(x',t) \big) dy' ds. 
\end{split}
\end{align*}
 From the second identity of \eqref{1006-3}, we observe
\begin{align}\label{0819-1-2}
\begin{split}
w^{L}_{i}(x,t) &= 4 D_{x_n} D_{x_i}\int_0^{x_n} \int_{\Rn} N(x-y) k_n(y,t)dy - 2R'_i k_n(x,t), \\
w_{ij}(x,t) &= 4 D_{x_i} D_{x_j} \int_0^{x_n} \int_{\Rn} N(x-y) k_j(y,t)dy,
\end{split}
\end{align}
where   $w_{ij}$ are defined in \eqref{260409-1} and 
 \begin{align*}%\label{160224-7}
k_j(x,t) = \int_0^t \int_{\Rn} D_{x_n} \Ga(x' -y', x_n, t-s) g_j(y',s) dy' ds, \quad 1 \leq j \leq n.
\end{align*}

Note that 
\begin{align}\label{lapl}
\De_x \int_0^{x_n} \int_{\Rn} N(x-y) k_j (y,t)dy = \frac12 k_j (x,t) + I(D_{x_n} k_j(\cdot, x_n, t))(x'),
\end{align}
where $Ik_j(\cdot, x_n,t)(x') = \int_{\Rn} N(x'-y', 0) k_j (y', x_n,t) dy'$.   Since $D_{x_i} I = \frac12 R'_i$ for $1 \leq i \leq n-1$,  from  \eqref{0819-1-2} and \eqref{lapl}, we have
\begin{align} \label{260409-2}
\begin{split}
 D_{x_n} w^{L}_{i}(x,t) &= 4\Big( - D_{x_i} \De_{x'} \int_0^{x_n} \int_{\Rn} N(x-y)k_n(y,t)dy + \frac12 D_{x_i} k_n(x,t) \\
 & \quad + D_{x_n} D_{x_i} I(k_n(\cdot, x_n))(x') \Big)  - 2 D_{x_n} R'_i k_n(x,t) \\
&= -4 D_{x_i} \De_{x'} \int_0^{x_n} \int_{\Rn} N(x-y)k_n(y,t)dy + 2 D_{x_i} k_n(x,t).
\end{split}
\end{align}
Note that  for $ y'\in \Rn$, we have 
\begin{align}\label{260227-1}
\begin{split}
|D_{y_i} \De_{y'}  k_n(y,t)| %&  \leq c  \int_0^t \int_{\Rn}  \frac{y_n}{(t -s)^{\frac{n+2}2}}  e^{-\frac{|y'|^2}{t -s}}      |D_{z_i} \De_{z'}   g_n(z',s)| dz'ds\\
& \leq c   \|D_{z_i} \De_{z'}   g_n\|_{L^\infty (\Rn \times (0, \infty))} \int_0^t \int_{\Rn}  \frac{y_n}{(t -s)^{\frac{n+2}2}}  e^{-\frac{|y' -z'|^2 + y_n^2}{t -s}}      dz'ds\\
& \leq c   \|D_{z_i} \De_{z'}   g_n\|_{L^\infty (\Rn \times (0, \infty))}
\end{split}
\end{align}
and  since $supp \, g_n \subset B'(0,1) \times (0,2)$, for $|y'|> 2$, we have 
\begin{align}\label{260227-2}
\begin{split}
|D_{y_i} \De_{y'}  k_n(y,t)|  % &  \leq c  \int_0^t \int_{\Rn}  \frac{y_n}{(t -s)^{\frac{n+2}2}}  e^{-\frac{|y'|^2}{t -s}}      |D_{z_i} \De_{z'}   g_n(z',s)| dz'ds\\
& \leq c   \|D_{z_i} \De_{z'}   g_n\|_{L^\infty (\Rn \times (0, \infty))} \int_0^t \int_{B'(0,1)}  \frac{y_n}{(t -s)^{\frac{n+2}2}}  e^{-\frac{|y'|^2 + y_n^2}{t -s}}      dz'ds\\
& \leq c   \|D_{z_i} \De_{z'}   g_n\|_{L^\infty (\Rn \times (0, \infty))} \int_0^t   \frac{y_n}{(t -s)^{\frac{n+2}2}}  e^{-\frac{|y'|^2 + y_n^2 }{t -s}}      ds\\
& \leq c   \|D_{z_i} \De_{z'}   g_n\|_{L^\infty (\Rn \times (0, \infty))}   y_n     |y'|^{-n-1}.
\end{split}
\end{align}
From \eqref{260227-1} and \eqref{260227-2}, we have $4 D_{x_i} \De_{x'} \int_0^{x_n} \int_{\Rn} N(x-y)k_n(y)dy \rightarrow 0$ as $ x_n \rightarrow 0$.  Since 
\[
D_{x_i} k_n(x', 0, t) = \int_0^t \int_{\Rn} D_{x_n} \Ga(x' -z', x_n, t-s) D_{z_i} g_n(z',s) dz'ds \Big|_{x_n =0} = -\frac12 D_{x_i} g_n(x',t),
\]
from \eqref{260409-2}, we have 
\begin{align}\label{w-li}
D_{x_n} w^L_{i}(x',0,t) &= -D_{x_i} g_n(x',t).
\end{align}  

Next, from $\eqref{1006-3}_2$,  we have 
\begin{align*}
D_{x_n} w_{ij}(x,t) &= 4 D_{x_i} \int_0^t \int_{\Rn} L_{nj}(x'-y', x_n, t-s) D_{y_i} g_j(y',s) dy' ds + 2 D_{x_i} R'_j k_j(x,t).
\end{align*}
Since $L_{nj}(x'-y',0, t-s) = 0$, we have 
\begin{align*}
D_{x_n} w_{ij}(x',0,t) &= -2 D_{x_i}R'_j g_j(x',t).
\end{align*}
Thus, it follows from \eqref{w-in} and \eqref{w-li} that $D_{x_n} w^L_{i}(x',0,t) +D_{x_n} w^N_i(x', 0,t) =0$.
Since $w_i$ is represented by \eqref{repre0926}$_1$, summing all the estimates completes the proof of  \eqref{normal}.

Next, we prove \eqref{representationpressure-1}.  Note that 
\begin{align}\label{1027-1}
\begin{split}
\int_{-\infty}^t  D_{x_j}\Ga(x,t-s) ds = - D_{x_j}N(x),\quad 1 \leq j \leq n, \quad \int_{\Rn} D_{x_n} N(y', x_n) N(x'-y',0) dy'=  \frac12 N(x).
% \int_{\Rn} D_{x_i} N(y', x_n) N(x'-y',0) dy'= R'_i N( x', x_n).
\end{split}
\end{align}

From $\eqref{1027-1}_1$, for $t > 0$, we have 
\begin{align}\label{0224-1}
\begin{split}
& \int_{-\infty}^t \int_{\Rn} g_j (z',s) \int_{\Rn}  N(y', x_n) D_{x_j}\Ga(x'-y' -z',0,t-s) dy' dz' ds\\
& = \int_{-\infty}^t \int_{\Rn}\big( g_j (z',s) -g_j(z',t) \big) 
  \int_{\Rn}   N(y', x_n) D_{x_j} \Ga(x'-y' -z',0,t-s) dy' dz' ds \\
&\quad -   \int_{\Rn} g_j (z',t) \int_{\Rn}  N(y', x_n)  D_{x_j} N(x'-y' -z',0) dy' dz'.
\end{split}
\end{align}
Notice that from $\eqref{1027-1}_2$, we note that
\[
\De' \int_{\Rn}   N(y', x_n) D_{x_j} N(x'-y' -z',0) dy' = -\frac12 D_{x_n} D_{x_j} N(x'-z', x_n).
\]
Hence, it follows from \eqref{0224-1} that
\begin{align}\label{p-est}
\begin{split}
& (D_t -\De') \int_{-\infty}^t \int_{\Rn} g_j (z',s) \int_{\Rn} D_{y_j} N(y', x_n) \Ga(x'-y' -z',0,t-s) dy' dz' ds\\
%& = (D_t -\De') \int_{-\infty}^t \int_{\Rn}\big( g_j (z',s) -g_j(z',t) \big) \int_{\Rn} D_{y_j} N(y', x_n) \Ga(x'-y' -z',0,t-s) dy' dz'ds \\
%&\quad - (D_t -\De') \int_{\Rn} g_j (z',t) \int_{\Rn} D_{y_j} N(y', x_n) N(x'-y' -z',0) dy' dz'\\
& = - \int_{-\infty}^t \int_{\Rn} D_t g_j(z',t) \int_{\Rn} D_{y_j} N(y', x_n) \Ga(x'-y' -z',0,t-s) dy' dz' ds \\
&\quad + \int_{-\infty}^t \int_{\Rn}\big( g_j (z',s) -g_j(z',t) \big) 
  \int_{\Rn} D_{y_j} N(y', x_n) (D_t -\De') \Ga(x'-y' -z',0,t-s) dy' dz' ds \\
&\quad - \int_{\Rn} D_t g_j (z',t) \int_{\Rn} D_{y_j} N(y', x_n) N(x'-y' -z',0) dy' dz'\\
&\quad + \De' \int_{\Rn} g_j (z',t) \int_{\Rn} D_{y_j} N(y', x_n) N(x'-y' -z',0) dy' dz'\\
& = \int_{-\infty}^t \int_{\Rn}\big( g_j (z',s) -g_j(z',t) \big) 
  \int_{\Rn} D_{y_j} N(y', x_n) \Big(-\frac{1}{2(t -s)}\Big) \Ga(x'-y' -z',0,t-s) dy' dz' ds \\
&\quad -  \frac12 D_{x_j} D_{x_n} \int_{\Rn} g_j (z',t) N(x' -z',x_n) dz'.
\end{split}
\end{align}

Recalling \eqref{representationpressure} and using \eqref{p-est}, we have
 \begin{align*}%\label{representationpressure-1-1}
\begin{split}
p(x,t) 
& = -4 \sum_{j =1}^{n} D_{x_j} D_{x_n} \int_{\Rn} N(x' -y', x_n) g_j(y', t) dy' - 2 D_t N*' g_n (x,t) \\
& \quad - 2\sum_{j =1}^{n} \int_{-\infty}^t \int_{\Rn}\frac{g_j (z',s) -g_j(z',t)}{t -s} 
  \int_{\Rn} D_{y_j} N(y', x_n) \Ga(x'-y' -z',0,t-s) dy' dz' ds. 
\end{split}
\end{align*}
We complete the proof of \eqref{representationpressure-1}.
\qed 
 
%\section{Proof of Lemma \ref{lemma1116-1}}
%\setcounter{equation}{0}\label{appendix}
%\label{prooflemma3-1}.

\subsection{Proof of Lemma \ref{lemma1116-1}}
\label{prooflemma3-1}

Using the integration by parts, we have 
\begin{align*}
D_{x'}^\be  f(x',t)  = & \int_{\Rn}    \Ga'(x' -y', t)     \big(   D_{y'}^\be  R'_i \psi(y')   -  D_{x'}^\be   R'_i \psi(x')    \big)  dy'\\
=&\int_{|x' -y'| \leq \frac12 |x'|}  \cdots dy+\int_{|x' -y'| \geq  \frac12 |x'|}\cdots dy:=\tilde f_1(x',t) +\tilde f_2(x',t)
\end{align*}
%Let  
%\begin{align*}
%\tilde f_1(x',t) &  =  \int_{|x' -y'| \leq \frac12 |x'|}    \Ga'(x' -y', t)     \big(      D_{y'}^\be  R'_i \psi(y')   -     D_{x'}^\be  R'_i \psi(x')    \big)  dy',\\
%\tilde f_2(x',t) & = \int_{|x' -y'| \geq  \frac12 |x'|}    \Ga'(x' -y', t)     \big(    D_{y'}^\be   R'_i \psi(y')   -    D_{x'}^\be R'_i \psi(x')    \big)  dy'  
%\end{align*}
%such that $D_{x'}^\be  f = \tilde f_1 + \tilde f_2$. 
Since $ \int_{|y'| \leq \frac12 |x'|} \Ga'(y',t) y' dy' =0$, we have 
\begin{align*}%\label{260110-5}
\begin{split}
 \tilde  f_1(x',t)
 & = \int_{|y'| \leq \frac12 |x'|} \Ga'(y',t)   \int_0^1 \na_{x'}    D_{x'}^\be R'_i \psi (x' -\theta y') \cdot (-y') d \theta dy'\\
 & = \int_{|y'| \leq \frac12 |x'|} \Ga'(y',t)   \int_0^1 \Big(  \na_{x'}   D_{x'}^\be  R'_i \psi (x' -\theta y') - \na_{x'}  D_{x'}^\be R'_i \psi (x' ) \Big) \cdot (-y') d \theta dy'\\
 & = \int_{|y'| \leq \frac12 |x'|} \Ga'(y',t)   \int_0^1  \int_0^1 (-\te y')  \na^2_{x'}   D_{x'}^\be  R'_i \psi (x' -\eta \theta y')  (- y') d\eta   d \theta dy'\\
& := \tilde f_{11}(x',t) +\tilde f_{12}(x',t),
\end{split}
\end{align*}
where 
\begin{align*}
 \tilde f_{11} (x',t) & =  \int_{|y'| \leq \frac12 |x'|} \Ga'(y',t)   \int_0^1  \te \int_0^1     y'        \na^2_{x'}    D_{x'}^\be R'_i \psi (x'  )    y'      d\eta   d \theta dy',\\
 \tilde f_{12}(x',t)& = \int_{|y'| \leq \frac12 |x'|} \Ga'(y',t)   \int_0^1  \te  \int_0^1  y' \Big(      \na^2_{x'}   D_{x'}^\be R'_i \psi (x' -\eta \theta y') -     \na^2_{x'}   D_{x'}^\be R'_i \psi (x' )  \Big)   y' d\eta   d \theta dy'.
 \end{align*}
Since $
\int_{|y'| \leq \frac12 |x'|}  e^{-\frac{|y'|^2}t} y_k y_l dy'
  = \de_{kl}   \frac1{n-1} \int_{|y'| \leq \frac12 |x'|}  e^{-\frac{|y'|^2}t}   |y'|^2 dy'$, 
we have 
\begin{align*}
\begin{split}
 \tilde f_{11}(x',t) % & =   \int_{|y'| \leq \frac12 |x'|} \Ga'(y',t)   \int_0^1  \int_0^1    (\te y_k)       y_l      d\eta   d \theta dy'\\
& =  \De'      D_{x'}^\be   R'_i \psi (x'  )       \int_{|y'| \leq \frac12 |x'|} \Ga'( y',t)  |y'|^2 dy' \int_0^1 \int_0^1 \te d\te d\eta   \\
& = \frac1{2(n-1)}   \De'      D_{x'}^\be   R'_i \psi (x'  )  \int_{|y'| \leq \frac12 |x'|}   \Ga'(y',t) | y'|^2 dy'\\
& =  \frac1{2(n-1)}   \De'   D_{x'}^\be     R'_i \psi (x'  )   t  \int_{|y'| \leq \frac12\frac{ |x'|}{\sqrt{t}}} \Ga'( y',1)   | y'|^2   dy'\\
&=\frac1{2(n-1)}   \De'   D_{x'}^\be     R'_i \psi (x'  )   t  \bke{\int_{\Rn }\Ga'( y',1)   | y'|^2   dy' -\int_{|y'| \geq \frac12\frac{ |x'|}{\sqrt{t}}} \Ga'( y',1)   | y'|^2   dy'}\\
%& =  \frac1{2(n-1)}   \De'     D_{x'}^\be   R'_i \psi (x'  )    t  \int_{\Rn } \Ga'( y',1)   | y'|^2   dy' \\
%\\ & \qquad 
%-   \frac12 \frac1{n-1}   \De'    D_{x'}^\be  R'_i \psi (x'  )  t  \int_{|y'| \geq \frac12\frac{ |x'|}{\sqrt{t}}} \Ga'( y',1)   | y'|^2   dy'\\
& := \tilde f_{111}(x',t) + \tilde f_{112}(x',t).
\end{split}
\end{align*}
Since  $  e^{-\frac{|x'|^2}t} \leq c_m ( \frac{|x'|^2}t )^{-m}$ for  any $m > 0$, we have 
\begin{align*}
| \tilde f_{112}(x',t)| \leq  \frac1{2(n-1)} | \De'    D_{x'}^\be   R'_i \psi (x'  ) | t \int_{|y'| \geq \frac12\frac{ |x'|}{\sqrt{t}}} \Ga'( y',1)   | y'|^2   dy'
\leq   c    | \De'    D_{x'}^\be   R'_i \psi (x'  ) |  t   e^{-\frac{|x'|^2}t}.
\end{align*}

On the other hand, for  $\tilde f_{12}$ direct calculations show that
\begin{align*}
|\tilde f_{12}(x',t)| & =  \left| \int_{|y'| \leq \frac12 |x'|} \Ga'(y',t) \int_0^1 \int_0^1 \te y_k y_l \Big( D_{x_k} D_{x_l}   D_{x'}^\be R'_i \psi (x' -\eta \theta y') - D_{x_k} D_{x_l}   
                                             D_{x'}^\be R'_i \psi (x') \Big) d\eta d\theta dy' \right|\\
%& \leq c \int_{|y'| \leq \frac12 |x'|} \Ga'(y',t) |y'|^3 \frac{1}{|x'|^{n+2 +|\be| }} dy' \\%\quad \text{(assuming proper decay of $\nabla^3 R'_i \psi$)} \\
& \leq c \frac{1}{|x'|^{n+2 + |\be|}} \int_{|y'| \leq \frac12 |x'|} \Ga'(y',t) |y'|^3 dy' \\
%= c \frac{t^{\frac32}}{|x'|^{n+2 +|\be|}} \int_{|y'| \leq \frac12 \frac{|x'|}{\sqrt{t}}} \Ga'(y',1) |y'|^3 dy' \\
& \leq c \frac{t^{\frac32}}{|x'|^{n+2 +|\be|}}. %\quad \text{(since $|x'| \geq c_* \geq \sqrt{t} $)}
\end{align*}
where we used that $|x'| \geq c_* \geq \sqrt{t} $.
Lastly, for $\tilde f_2$ it is straightforward that 
\begin{align*}
\tilde f_2(x',t) & \leq     2  \| D^\be_{x'} R'_i \psi\|_{L^\infty (\R)} \int_{|x' -y'| \geq \frac12|x'|} \Ga' (x' -y', t)  dx \leq c e^{-\frac{|x'|^2}t}\leq   c \frac{t^{\frac32+\frac{|\be|}2 }}{|x'|^{n+2 +|\be|}}.
\end{align*}
Since $\frac1{2(n-1)}  \int_{\Rn }\Ga'( y',1)   | y'|^2   dy'  =1$, taking   $ F_1 = \tilde f_{111}$ and $ F_2 =  \tilde f_{112}   +  \tilde f_{12}  + \tilde f_2$,  we obtain the result of  Lemma \ref{lemma1116-1}.
\qed

%\section{Proof of Theorem \ref{so}}\label{existenceofsolution2} 
%\setcounter{equation}{0}

\subsection{Proof of Lemma \ref{so}}\label{existenceofsolution2}

The cases $k\ge 2$ and $k=0$ were proved in \cite{S02} and \cite{CJ2}, respectively. To the best of the authors' knowledge, the case $k=1$ has not yet been treated in the literature. For the reader's convenience, we therefore include a proof for this case.

We recall that the solution $w$ to the system \eqref{Stokes-10-2} with boundary data $g$ is represented by \eqref{repre0926}.
By the well-known estimates for solutions to the heat equation (see e.g. \cite[Section 2 of Chapter IV]{LS}), for $l=0,1,2$ and $0<\alpha<1$, we have
\begin{align*}
\begin{split}
\| w^G_i \|_{C^{l+\al, \frac{l}2 +\frac{\al}2}(Q_+(T))} &\leq c \| g_j \|_{C^{l+\al, \frac{l}2 +\frac{\al}2}(q(T))}.
\end{split}
\end{align*}
From \eqref{NB0926}, for $l = 0, \, 1, \, 2$ and $0 < \al < 1$, similar arguments imply that
\begin{align*}
\begin{split}
\| w^N_i \|_{C^{l+\al, \frac{l}2 +\frac{\al}2}(Q_+(T))} &\leq c \| R'_i g_n \|_{C^{l+\al, \frac{l}2 +\frac{\al}2}(q(T))}, \\
\| w^B_i \|_{C^{l+\al, \frac{l}2 +\frac{\al}2}(Q_+(T))} &\leq c \| R'_i g_n \|_{C^{l+\al, \frac{l}2 +\frac{\al}2}(q(T))}.
\end{split}
\end{align*}
In Theorem 3.1 of \cite{CJ2}, the authors showed the following result:
\begin{align*}%\label{1127-1}
\| w_{ij} \|_{C^{\al, \frac{\al}2}(Q_+(T))} \leq c \| g_j \|_{C^{\al, \frac{\al}2}(q(T))}, \quad 1 \leq i \leq n, \quad 1 \leq j \leq n-1,
\end{align*} 
where $w_{ij}$ are defined in \eqref{260409-1}.

Note that 
\begin{align}\label{1127-2}
\| k_j \|_{C^{\al, \frac{\al}2}(Q_+(T))} \leq \| g_j \|_{C^{\al, \frac{\al}2}(q(T))}.
\end{align}

Hence, from \eqref{0819-1-2} and \eqref{260409-2}, for $i \neq n$, we have 
\begin{align}\label{0819-1-1}
\begin{split}
D_{x'} w_{i}^L(x,t) &= 4 D_{x_n} D_{x_i} \int_0^{x_n} \int_{\Rn} N(x-y) D_{y'} k_n(y,t) dy - 2 D_{x'} R'_i k_n(x,t).
\end{split}
\end{align}
Thus, from \eqref{260409-2}, \eqref{1127-2} and \eqref{0819-1-1}, we obtain 
\begin{align}\label{1127-3}
\| D_{x} w_{i}^L \|_{C^{\al,\frac{\al}2}(Q_+(T))} \leq c \| D_{x'} g_n \|_{C^{\al,\frac{\al}2}(q(T))}.
\end{align}
We recall the following estimate (see \cite[Section 3]{KS} and  \cite[Section 2]{So}):
\begin{align}\label{estimateso}
|D^k_{x'} L_{ij}(x,t)| \leq c_\la \frac{x_n^\la}{t^{\frac12+\frac{\la}2}(|x|^2 + t)^{\frac{n}2}}, \quad \forall \la \in (0, 1), \quad k \in \mathbb{N} \cup \{0\}, \quad 1 \leq i \leq n, \,\, 1 \leq j \leq n-1.
\end{align}
Now, we will show the H\"{o}lder continuity of $w_i^L$ with respect to time. For $0 < \tau < t$, 
\begin{align*}
 w^{L}_{i}(x,t) - w^{L}_{i}(x,\tau) &= \int_0^\tau \int_{\Rn} L_{ni}(x'-y', x_n, s) \Big( g_n(y',t-s) - g_n(y', \tau-s) \Big) dy' ds \\
&\quad + \int_\tau^t \int_{\Rn} L_{ni}(x'-y', x_n, s) g_n(y',t- s) dy' ds \\
&= I_1 + I_2.
\end{align*}
From \eqref{estimateso}, we have 
\begin{align*}
\begin{split}
I_1 &\leq c \| g_n \|_{L^\infty(\Rn; C^{\frac12 + \frac{\al}2}(0, T))} (t -\tau)^{\frac12 + \frac{\al}2} \int_0^\tau \int_{\Rn} s^{-\frac12 -\frac{\la}2} \frac{x_n^\la}{(|x-y'|^2 + s)^{\frac{n}2}} dy' ds \\
&\leq c \| g_n \|_{C^{1 +\al, \frac12 + \frac{\al}2}(q(T))} (t -\tau)^{\frac12 + \frac{\al}2} \int_0^\tau s^{-\frac12 -\frac{\la}2} x_n^\la (x_n^2 + s)^{-\frac12} ds \\
&\leq c \| g_n \|_{C^{1 +\al, \frac12 + \frac{\al}2}(q(T))} (t -\tau)^{\frac12 + \frac{\al}2} \int_0^{\frac{\tau}{x_n^2}} s^{-\frac12 -\frac{\la}2} (1 + s)^{-\frac12} ds \\
&\leq c \| g_n \|_{C^{1 +\al, \frac12 + \frac{\al}2}(q(T))} (t -\tau)^{\frac12 + \frac{\al}2}.
\end{split}
\end{align*}
Since $g(y', 0) = 0$, we have 
\begin{align}\label{0104-1}
I_2 &\leq c \| g_n \|_{C^{1 +\al, \frac12 + \frac{\al}2}(q(T))} \int_\tau^t (t -s)^{\frac12 + \frac{\al}2} s^{-\frac12 -\frac{\la}2} x_n^\la (x_n^2 + s)^{-\frac12} ds.
\end{align}
Note that for $t, x_n > 0$, 
\begin{align}\label{0104-2}
 t^{\frac12 -\frac{\la}2} x_n^\la (x_n^2 + t)^{-\frac12} \leq c x_n^\la (x_n^2 + t)^{-\frac{\la}2} \leq c.
\end{align}
If $\frac{t}2 < \tau < t$, then from \eqref{0104-1} and \eqref{0104-2}, we have 
\begin{align*}
\begin{split}
I_2 &\leq c \| g_n \|_{C^{1 +\al, \frac12 + \frac{\al}2}(q(T))} \int_\tau^t (t -s)^{\frac12 + \frac{\al}2} t^{-\frac12 -\frac{\la}2} x_n^\la (x_n^2 + t)^{-\frac12} ds \\
&\leq c \| g_n \|_{C^{1 +\al, \frac12 + \frac{\al}2}(q(T))} (t -\tau)^{\frac12 + \frac{\al}2} t^{\frac12 -\frac{\la}2} x_n^\la (x_n^2 + t)^{-\frac12} \\
&\leq c \| g_n \|_{C^{1 +\al, \frac12 + \frac{\al}2}(q(T))} (t -\tau)^{\frac12 + \frac{\al}2}.
\end{split}
\end{align*}

If $0 < \tau \leq \frac{t}2$, then we split the integral as
\begin{align*}
\begin{split}
I_2 &\leq c \| g_n \|_{C^{1 +\al, \frac12 + \frac{\al}2}(q(T))} \int_\tau^{\frac{t}2} (t -s)^{\frac12 + \frac{\al}2} s^{-\frac12 -\frac{\la}2} x_n^\la (x_n^2 + s)^{-\frac12} ds \\
&\quad + c \| g_n \|_{C^{1 +\al, \frac12 + \frac{\al}2}(q(T))} \int_{\frac{t}2}^t (t -s)^{\frac12 + \frac{\al}2} s^{-\frac12 -\frac{\la}2} x_n^\la (x_n^2 + s)^{-\frac12} ds \\
&= I_{21} + I_{22}.
\end{split}
\end{align*}
Since $t < 2(t -\tau)$ for $0 < \tau < \frac{t}2$, we have 
\begin{align*}
\begin{split}
I_{21} &\leq c \| g_n \|_{C^{1 +\al, \frac12 + \frac{\al}2}(q(T))} \int_0^{\frac{t}2} t^{\frac12 + \frac{\al}2} s^{-\frac12 -\frac{\la}2} x_n^\la (x_n^2 + s)^{-\frac12} ds \\
&\leq c \| g_n \|_{C^{1 +\al, \frac12 + \frac{\al}2}(q(T))} t^{\frac12 + \frac{\al}2} \int_0^{\frac{t}{2x_n^2}} s^{-\frac12 -\frac{\la}2} (1 + s)^{-\frac12} ds \\
&\leq c \| g_n \|_{C^{1 +\al, \frac12 + \frac{\al}2}(q(T))} (t-\tau)^{\frac12 + \frac{\al}2},
\end{split}
\end{align*}
and
\begin{align*}
\begin{split}
I_{22} &\leq c \| g_n \|_{C^{1 +\al, \frac12 + \frac{\al}2}(q(T))} \int_{\frac{t}2}^t (t -s)^{\frac12 + \frac{\al}2} t^{-\frac12 -\frac{\la}2} x_n^\la (x_n^2 + t)^{-\frac12} ds \\
&\leq c \| g_n \|_{C^{1 +\al, \frac12 + \frac{\al}2}(q(T))} t^{\frac32 + \frac{\al}2} t^{-\frac12 -\frac{\la}2} x_n^\la (x_n^2 + t)^{-\frac12} \\
&\leq c \| g_n \|_{C^{1 +\al, \frac12 + \frac{\al}2}(q(T))} (t-\tau)^{\frac12 + \frac{\al}2}.
\end{split}
\end{align*}
Hence, we obtain
\begin{align}\label{0412-10}
\| w^L_i \|_{L^\infty(\R_+; \dot{C}^{\frac12 +\frac{\al}2}(Q_+(T)))} \leq c \| g_n \|_{C^{1 +\al, \frac12 + \frac{\al}2}(q(T))} (t-\tau)^{\frac12 + \frac{\al}2}. 
\end{align}
Summing all the estimates \eqref{1127-3} and \eqref{0412-10}, we have 
\begin{align*}
\| w_{i}^L \|_{C^{1+\al, \frac12 + \frac{\al}2}(Q_+(T))} \leq c \Big( \| g_n \|_{C^{1+\al, \frac12 +\frac{\al}2}(q(T))} + \| R'_i g_n \|_{C^{1+\al, \frac12 +\frac{\al}2}(q(T))} \Big).
\end{align*}
Similarly, we have 
\begin{align*}
\| w_{ij} \|_{C^{1+\al, \frac12 + \frac{\al}2}(Q_+(T))} \leq c\| g \|_{C^{1+\al, \frac12 +\frac{\al}2}(q(T))}.
\end{align*}
This completes the proof of (1) of Theorem \ref{so}.
  
Similarly, we obtain
\begin{align*}
\begin{split}
\| D_{x}^2 w_{i}^L \|_{C^{\al,\frac{\al}2}(Q_+(T))} &\leq c \Big( \| D^2_{x'} g_n \|_{C^{\al,\frac{\al}2}(q(T))} + \| D^2_{x'}R' g_n \|_{C^{\al,\frac{\al}2}(q(T))} \Big), \\ 
\| D_t w_{i}^L \|_{C^{\al,\frac{\al}2}(Q_+(T))} &\leq c \Big( \| D_t g_n \|_{C^{\al,\frac{\al}2}(q(T))} + \| D_t R' g_n \|_{C^{\al,\frac{\al}2}(q(T))} \Big).
\end{split}
\end{align*}
Therefore, we conclude that 
\begin{align*}
\| w_{i}^L \|_{C^{2 +\al, 1 +\frac{\al}2}(Q_+(T))} \leq c \Big( \| g_n \|_{C^{2+\al, 1 + \frac{\al}2}(q(T))} + \| R' g_n \|_{C^{2+\al, 1 + \frac{\al}2}(q(T))} \Big).
\end{align*}
This completes the proof  of (2)of Lemma \ref{so}.
\qed

\subsection{Proof of Theorem \ref{theo1123}  }\label{prooftheo1123}
 
Since $\int_{\Rn }  y_k \Ga'(y', t-s)   dy' =0 $ for all $ 1 \leq k \leq n-1$ and $t -s> 0$,
applying the mean value theorem twice,   we have 
\begin{align*}
\begin{split}
&\int_{\Rn} \Ga' (x' -y',t-s)\big(R'_i g_n(y',s)   - R'_i g_n(x',s)  \big)   dy' \\
  =&\int_{\Rn} \Ga' (x' - y',t-s)    \int_0^1  \na'R'_i g_n (\te y' + (1 -\te)x',s)  d\te  \cdot  (y' -x')   dy'\\
= &\int_{\Rn} \Ga' (x' - y',t-s)   \int_0^1  \Big(  \na'R'_i g_n (\te y' + (1 -\te)x',s)  -   \na'R'_i g_n ( x',s) \Big)  d\te  \cdot  (y' -x')    dy'\\
\leq &   \| R' g\|_{L^\infty(0, t; \dot C^{1 +\al}(\Rn))} \int_{\Rn} \Ga' (x' - y',t-s)   \int_0^1  \te   |y' -x'|^{1 +\al} d\te     dy'\\
%& \leq  c \| R' g\|_{L^\infty(0, T; \dot C^{1 +\al}(\Rn))}   \int_{\Rn } |y'|^{1 +\al} \Ga' (y',t-s)  dy'\\
\leq & c  \| R' g\|_{L^\infty(0, t; \dot C^{1 +\al}(\Rn))} (t-s)^{\frac12 +\frac{\al}2}.
\end{split}
\end{align*}
Direct calculations show that 
\begin{align*}
\begin{split}
&\int_{\Rn} \Ga' (x' -y',t-s)\big(R'_i g_n(x',s)   - R'_i g_n(x',t)  \big)   dy' \\
 \leq   & \| R' g_n\|_{L^\infty(\Rn; \dot C^{\frac12 +\frac{\al}2}(0,t))} \int_{\Rn} \Ga' (x' - y',t-s)(t -s)^{\frac12 +\frac{\al}2}      dy'\\
%& \leq  c \| R' g\|_{L^\infty(0, T; \dot C^{1 +\al}(\Rn))}   \int_{\Rn } |y'|^{1 +\al} \Ga' (y',t-s)  dy'\\
\leq &c    \| R' g_n\|_{L^\infty(\Rn; \dot C^{\frac12 +\frac{\al}2}(0,t))}(t-s)^{\frac12 +\frac{\al}2}.
\end{split}
\end{align*}
Hence, we obtain
 \begin{align*}
 \begin{split}
& \int_0^t  \frac1{(t -s)^\frac32}    \int_{\Rn}   \Ga' (x' -y', t-s) \big( R'_i g_n(y',s)   -  R'_i  g_n(x',s)    \big)  dy' ds \\
\leq & c   \| R' g_n\|_{L^\infty(0, t;  \dot C^{1 +\al}(\Rn))}    \int_0^t (t -s)^{-1 +\frac{\al}2} ds\\
\leq & c   \| R' g_n\|_{L^\infty(0, t;  \dot C^{1 +\al}(\Rn))} t^{\frac{\al}2}.
  \end{split}
  \end{align*}
  Since  $ R'_j  g_n (x',0)   =0$, we note that $|R'_j  g_n (x',t)  - R'_j  g_n (x',0)  | \leq \| R'_j  g_n\|_{L^\infty(\Rn; \cdot C^{\frac12 +\frac{\al}2}(0, t))} t^{\frac12 +\frac{\al}2}$, and therefore we observe that 
   \begin{align*}
  \begin{split}
& \int_{-\infty}^0 \frac1{(t -s)^\frac32}    \int_{\Rn}   \Ga' (x' -y', t-s)      R'_i  g_n(x',t)    dy' ds\\
=& \| R'_i g_n \|_{L^\infty(\Rn;\cdot C^{\frac12 +\frac{\al}2} (0, t))}  t^{\frac12 +\frac{\al}2} \int_{-\infty}^0 \frac1{(t -s)^\frac32}    \int_{\Rn}   \Ga' (x' -y', t-s)          dy' ds\\
\leq  & c    \| R' g_n\|_{L^\infty(\Rn; \dot C^{\frac12 +\frac{\al}2}(0,t))} t^{\frac{\al}2}.
  \end{split}
 \end{align*}
It follows ffrom the Holder continuity of $R' g_n$ with respect to $t$ that 
  \begin{align*}
  \begin{split}
 & \int_0^t  \frac1{(t -s)^\frac32}    \int_{\Rn}   \Ga' (x' -y', t-s) \big( R'_i g_n(x',s)   -  R'_i  g_n(x',t)    \big)  dy' ds\\
 \leq & c  \| R' g_n\|_{L^\infty(\Rn; \dot C^{\frac12 +\frac{\al}2}(0,t))}   \int_0^t  \frac1{(t -s)^\frac32}    \int_{\Rn}   \Ga' (x' -y', t-s)  (t -s)^{\frac12 +\frac{\al}2}  dy' ds\\
 \leq & c    \| R' g_n\|_{L^\infty(\Rn; \dot C^{\frac12 +\frac{\al}2}(0,t))} t^{\frac{\al}2}.
  \end{split}
 \end{align*}
Hence, we have 
   \begin{align*}
  \begin{split}
 \int_{-\infty}^t \frac1{(t -s)^\frac32}    \int_{\Rn}   \Ga' (x' -y', t-s)    \big(   R'_i  g_n(x',t)  - R'_i g_n (x',s) \big)   dy' ds
   \leq  c    \| R' g_n\|_{L^\infty(\Rn; \dot C^{\frac12 +\frac{\al}2}(0,t))} t^{\frac{\al}2}.
  \end{split}
 \end{align*}
Similarly, we have 
 \begin{align*}
 \int_0^t  \frac1{(t -s)^\frac32}    \int_{\Rn}   \Ga' (x' -y', t-s) \big(  g_i(y',s)   -   g_i(x',s)    \big)  dy' ds 
  & \leq  c   \| g\|_{L^\infty(0, t;  \dot C^{1 +\al}(\Rn))} t^{\frac{\al}2},\\
   \int_{-\infty}^t  \frac1{(t -s)^\frac32}    \int_{\Rn}   \Ga' (x' -y', t-s) \big(  g_i(x',s)   -   g_i(x',t)    \big)  dy' ds 
  & \leq  c  \| g\|_{L^\infty(\Rn; \dot C^{\frac12 +\frac{\al}2}(0,t))} t^{\frac{\al}2}.
 \end{align*}
Since  $  R'_i g_j (x',0)     =0$, we have 
\begin{align*}
\begin{split}
 | D_{x_i}R'_j  g_j (x',t) | & \leq   c \| D_{x'} R'_j g\|_{L^\infty (\Rn ; C^{\frac{\al}2}(0, t)) }  t^{\frac{\al}2}.
%    |D_{x_i} g_n(x',t)|  & \leq c \| D_{x'}   g_n\|_{L^\infty (\Rn ; C^{\frac{\al}2}(0, t))} t^{\frac{\al}2}.
    \end{split}
\end{align*}
Then, from Theorem \ref{thoe1021-1} , we have
\begin{align*}  
\begin{split}
| D_{x_n} w_i   (x',0,t)|  & \leq c  t^{\frac{\al}2} \Big(  \|  g\|_{ C^{1 +\al,\frac12 +\frac{\al}2}(\Rn \times (0, t))}    +  \| R' g\|_{ C^{1 +\al,\frac12 +\frac{\al}2}(\Rn \times (0, t))} \Big).
\end{split}
\end{align*}
This completes the proof of  Lemma \ref{theo1123}. 
\qed

% \section{Proof of Proposition \ref{prop1120} }
%\setcounter{equation}{0}
%\label{proofprop1120}

\subsection{Proof of Lemma \ref{prop1120} }
\label{proofprop1120}

From Proposition 4.4 in  \cite{CJ2} and its proof, we recall that
\begin{align*}%\label{1118-3}
\begin{split}
\|   V^1|_{x_n =0}\|_{L^\infty(0, T; \dot C^\al(\Rn))} &  \leq cT^{\frac12 +\frac{\al}2} \|   f \|_{L^\infty(0, T; L^\infty(0, \infty; \dot C^\al(\Rn)))},\\
\| V^1|_{x_n =0}  \|_{L^\infty( \Rn; \dot C^{\frac12 +\frac{\al}2}(0, T))} &  \leq c   \| f \|_{L^\infty(0, T; L^\infty(0, \infty; \dot C^\al(\Rn)))},\\
\| V^1|_{x_n =0} \|_{L^\infty (\Rn \times (0, T))} &  \leq cT^{\frac12 +\frac{\al}2} \| f \|_{L^\infty(0, T; L^\infty(0, \infty; \dot C^\al(\Rn)))}.
\end{split}
\end{align*}
Since $ R': \dot C^\al(\Rn) \rightarrow \dot C^\al(\Rn)$ is bounded,   it follows that
\begin{align*}%\label{1118-3-5}
\begin{split}
\|  R' V^1|_{x_n =0}\|_{L^\infty(0, T; \dot C^\al(\Rn))} &  \leq cT^{\frac12 +\frac{\al}2} \| R'  f \|_{L^\infty(0, T; L^\infty(0, \infty; \dot C^\al(\Rn)))} \\
&  \leq cT^{\frac12 +\frac{\al}2}  \| f \|_{L^\infty(0, T; L^\infty(0, \infty; \dot C^\al(\Rn)))}.
\end{split}
\end{align*}
Similarly, we have 
\begin{align*}
\begin{split}
\| R'  V^1|_{x_n =0}  \|_{L^\infty( \Rn; \dot C^{\frac12 +\frac{\al}2}(0, T))} &  \leq c  \|  f \|_{L^\infty(0, T; L^\infty(0, \infty; \dot C^\al(\Rn)))},\\
\| R' V^1|_{x_n =0} \|_{L^\infty (\Rn \times (0, T))} &  \leq cT^{\frac12 +\frac{\al}2} \|  f \|_{L^\infty(0, T; L^\infty(0, \infty; \dot C^\al(\Rn)))}.
\end{split}
\end{align*}
We complete the proof of Lemma \ref{prop1120}.
\qed

%\section{Proof of Proposition \ref{thoe1021-1}}\label{prooflemma1}
%\setcounter{equation}{0}

%\section{Proof of Theorem \ref{theo1123}  }
%\setcounter{equation}{0}
%\label{prooftheo1123}
% 

%\section{ }
%\setcounter{equation}{0}
%\label{proofprop0320-2}

\subsection{Conditions of $\phi$ in Assumption \ref{assume2} and Assumption \ref{assume3}}\label{proofprop0320-2}
 
 In this subsection, we provide some conditions of $\phi$ satisfying Assumption \ref{assume2} or Assumption \ref{assume3}.
\begin{lemma}\label{prop0320-1}
 Let $g$ be given by \eqref{0502-6} and let $\phi$ satisfy Assumption \ref{assume1}.  
 \begin{itemize}
 \item[(i)]
Suppose that there exists constant $r_0 \in (\frac32,2)$ such that  for all $ r_0 < t < 2$, 
 \begin{align} \label{260323-1}
\begin{split}
   2^\frac32 \phi(t) 
<  2^{-\frac32} \int_0^1   \phi(s)  ds. 
\end{split}
\end{align}
Then, $ M(t) < 0$ for all $ r_0 < t < 2$.

\item[(ii)]
Let $\displaystyle a_0 =\sup_{0 < t< 1} \phi'(t)$. Suppose that $\phi \in C^{1 +\frac{\al}2} ([0, 2 ) )$ satisfy   $\phi'(0)  =0$. Assume further that
there exist  $  r_0 \in (1,2)$ and $\de>0$  with $ r_0 +\de < 2$   such  that for $ r_0  < t < r_0+\de $, 
\begin{align*}%\label{260323-2}
-\frac12 \de^{-1} \phi(r_0) < \phi'(t) < -\frac{3}2 \sqrt{2} a_0 \de^{-\frac12}.
\end{align*}
Then,  $ M(r_0 +\de) < 0$.

\item[(iii)]
Suppose that $\phi \in C^{1 +\frac{\al}2} ([0, 2 ) )$ satisfy   $\phi(0) = \phi'(0)  =0$. Assume further that
 \begin{align}\label{260320-5}
 \inf_{1 < s<2} \phi'(s) \geq   - \frac14  \int_0^1 \phi'(s) (2 -s)^{-\frac12} ds.
 \end{align}
 Then, $ M(t) >    2^{-\frac12} \phi(1)$ for  all $ 1 < t<2$.
 \end{itemize}
 
\end{lemma}
\begin{proof}
Let $ \frac32 <t<2$.  Since $ \phi$ is decresing in $(1,2)$, we have 
\begin{align*} 
\begin{split}
M(t)   & \leq     \int_{-\infty}^1 \frac{\phi(t) -\phi(s)}{(t -s)^\frac32} ds 
=   2 \phi(t) (t -1)^{-\frac12}  
- \int_0^1 \frac{ \phi(s)}{(t -s)^\frac32} ds 
\leq  2^\frac32 \phi(t) 
-  2^{-\frac32} \int_0^1   \phi(s)  ds. 
\end{split}
\end{align*}
From \eqref{260323-1}, we have  $M(t) < 0$ for all $t \in (r_0, 2)$. This completes the proof of part (i).

By integration by parts, $M(t)$ can be rewritten as:
 \begin{align*}%\label{1026-6-1-1}
\begin{split}
M(t) = \int_{-\infty}^t  \frac{  \phi(t) - \phi(s)  }{(t-s)^\frac32}  ds  
%& =    2\int_{-\infty}^t \big( \phi(t) - \phi(s) \big)  \frac{d}{ds}   (t-s)^{-\frac12}  ds\\
& =   2\int_{0}^t   \phi' (s)     (t-s)^{-\frac12}  ds= 2 \int_0^1 \cdots ds + 2\int_1^t \cdots ds.
\end{split}
\end{align*}
Direct calculation for $1 < t < 2$ yields:
\begin{align*}
2\int_{0}^1   \phi' (s)     (t-s)^{-\frac12}  ds \leq 2a_0 \int_0^1    (t-s)^{-\frac12}  ds = 4a_0( t^\frac12 -(t-1)^\frac12 )
\leq 4 \sqrt{2}a_0.
\end{align*}
 Since  $ \phi$ is decreasing in $(1,2)$, for $ r_0 < t< r_0 +\de$,  we have 
\begin{align*}
2\int_1^{r_0 +\de}  \phi' (s)     (r_0 +\de -s)^{-\frac12}  ds  \leq -3 \sqrt{2} \de^{-\frac12} a_0 \int_{r_0}^{r_0 +\de}       (r_0 +\de-s)^{-\frac12}  ds
=- 6 \sqrt{2}  a_0.
\end{align*}
Additionally, observe that
\begin{align*}
\phi(r_0 +\de) =\int_{r_0}^{r_0 +\de} \phi'(s) ds +\phi(r_0) \geq  -\frac12 \de^{-1} \phi(r_0)\int_{r_0}^{r_0 +\de}   ds +\phi(r_0) 
=  -\frac12  \phi(r_0)  +\phi(r_0) =\frac12 \phi(r_0).
\end{align*}

Combining these estimates, we obtain $M(t) \leq 4\sqrt{2} a_0 - 6\sqrt{2} a_0 = -2\sqrt{2} a_0 < 0$, which completes the proof of part (ii).

Note that for $1 < t < 2$, \eqref{260320-5} implies:
\begin{align}\label{260323-3_revised}
\int_1^t \phi'(s) (t - s)^{-1/2} ds \geq \inf_{1 < s<2} \phi'(s) \int_1^t  (t - s)^{-1/2} ds\geq -\frac{1}{2} \int_0^1 \phi'(s) (t - s)^{-1/2} ds.
\end{align}
Substituting \eqref{260323-3_revised} into the expression for $M(t)$, we find:
\begin{align*}
M(t) &= 2 \int_0^1 \phi'(s) (t - s)^{-1/2} ds + 2 \int_1^t \phi'(s) (t - s)^{-1/2} ds \\
&\geq \int_0^1 \phi'(s) (t - s)^{-1/2} ds \quad ( (t-s)^{-\frac12} > ( 2-0)^{-\frac12} ) \\
&\geq 2^{-1/2} \int_0^1 \phi'(s) ds = 2^{-1/2} \phi(1).
\end{align*}
 This completes the proof of part (iii).

 \end{proof}
 
\section*{Acknowledgement}
T. Chang is supported by NRF grant  RS-2026-25473074 and  Kyungkeun Kang is supported by NRF grant  RS-2024-00336346 and RS-2024-00406821.

\end{document}